\documentclass[10pt]{article}

\usepackage{amsmath,amsfonts,amssymb} 

\usepackage{amsthm} 

\usepackage{graphicx,caption,subcaption}

\usepackage{fullpage}

\allowdisplaybreaks 

\begingroup
    \makeatletter
    \@for\theoremstyle:=definition,remark,plain\do{%
        \expandafter\g@addto@macro\csname th@\theoremstyle\endcsname{%
            \addtolength\thm@preskip\parskip
            }%
        }
\endgroup

\newtheorem{thm}{Theorem}
\newtheorem{lemma}[thm]{Lemma}
\newtheorem{cor}[thm]{Corollary}

\newtheorem{coro}[thm]{Corollary}
\newtheorem{construction}{Construction}[section]

\newtheorem{claim}{Claim}[section]

\newtheorem{definition}{Definition}[section]

\title{}
\author{}

\begin{document}

\title{\vspace{-0.5in} A hypergraph analog 
of Dirac's Theorem for long cycles\\ in 2-connected graphs, II: Large uniformities}

\author{
{{Alexandr Kostochka}}\thanks{
\footnotesize {University of Illinois at Urbana--Champaign, Urbana, IL 61801. E-mail: \texttt {kostochk@illinois.edu}.
 Research 
is supported in part by  NSF  Grant DMS-2153507.
}}
\and
{{Ruth Luo}}\thanks{
\footnotesize {University of South Carolina, Columbia, SC 29208, USA. E-mail: \texttt {ruthluo@sc.edu}.
}}
\and{{Grace McCourt}}\thanks{University of Illinois at Urbana--Champaign, Urbana, IL 61801, USA. E-mail: {\tt mccourt4@illinois.edu}. Research 
is supported in part by NSF RTG grant DMS-1937241.}}

\date{ \today}
\maketitle

\vspace{-0.3in}

\begin{abstract}
Dirac proved that each $n$-vertex $2$-connected graph with minimum degree $k$ contains a cycle  of length at least $\min\{2k, n\}$. We obtain analogous results for Berge cycles in hypergraphs. Recently, the authors proved an exact lower bound on the minimum degree ensuring a Berge cycle of length at least $\min\{2k, n\}$ in $n$-vertex $r$-uniform $2$-connected hypergraphs when $k \geq r+2$. 
In this paper we address the case $k \leq r+1$ in which the bounds have a different behavior. We prove that each $n$-vertex $r$-uniform $2$-connected hypergraph $H$ with minimum degree $k$ contains a Berge cycle of length at least $\min\{2k,n,|E(H)|\}$. If $|E(H)|\geq n$, this bound coincides with the bound of the  Dirac's Theorem for 2-connected graphs. 
%

\end{abstract}

\section{Introduction and Results}
\subsection{Definitions and known results for graphs}
 A hypergraph $H$ is a family of subsets of a ground set. We refer to these subsets as the {\em edges} of $H$ and to the elements of the ground set as the {\em vertices} of $H$. We use $E(H)$ and $V(H)$ to denote the set of edges and the set of vertices of $H$ respectively. We say $H$ is {\em $r$-uniform}  ($r$-graph, for short) if every edge of $H$ contains exactly $r$ vertices. A {\em graph} is a 2-graph. 

The {\em degree} $d_H(v)$ of a vertex $v$ in a hypergraph $H$ is the number of edges containing $v$. 
The {\em minimum degree}, $\delta(H)$, is the minimum over degrees of all vertices of $H$. 

A {\em hamiltonian cycle} in a graph is a cycle which visits every vertex. 
Degree conditions providing that a graph has a hamiltonian cycle were first studied  by
 Dirac in 1952.

\begin{thm}[Dirac~\cite{D}]\label{dirac}
Let $n \geq 3$. If $G$ is an $n$-vertex graph with minimum degree $\delta(G) \geq n/2$, then $G$ has a hamiltonian cycle.
\end{thm}

In the same paper, Dirac observed that if $\delta(G) \geq 2$, then $G$ has a cycle of length at least $\delta(G)+1$. This is best possible: examples are many copies of $K_{\delta(G)+1}$ sharing a single vertex. Each of these graphs has a cut vertex.  Dirac~\cite{D} showed that each 2-connected graph has a much longer cycle. 

\begin{thm}[Dirac~\cite{D}]\label{dirac2}
Let $n \geq k \geq 2$. If $G$ is an $n$-vertex, $2$-connected graph with minimum degree $\delta(G) \geq k$, then $G$ has a cycle of length at least $\min\{2k, n\}$.
\end{thm}

Both theorems are sharp. For example, when $2\leq k\leq n/2$, the complete bipartite graph $K_{k-1, n-(k-1)}$ has minimum degree
$k-1$ and the length of a longest cycle $2k-2$. Note that $K_{k-1, n-(k-1)}$ is $(k-1)$-connected, so for large $k$ demanding $3$-connectedness instead of $2$-connectedness   in Theorem~\ref{dirac2} will not strengthen the bound.
 
Stronger statements 
 for bipartite graphs  have been proved  by Voss and Zuluaga~\cite{VZ},
and then  refined by Jackson~\cite{jackson2}:

\begin{thm}[Jackson~\cite{jackson2}]\label{jack}
Let  $G$ be a 2-connected bipartite graph with bipartition $(A,B)$, where $|A|\geq |B|$. If each vertex of $A$ has degree at least $a$ and each vertex of $B$ has degree at least $b$,
 then $G$ has a cycle of length at least $2\min\{|B|, a+b-1, 2a-2\}$. Moreover, if $a=b$ and $|A|=|B|$, then $G$ has a cycle of length at least $2\min\{|B|,  2a-1\}$.
\end{thm}

A sharpness example for Theorem~\ref{jack} is a graph $G_3=G_3(a,b,a',b')$ for $a'\geq b'\geq a+b-1$ obtained from
disjoint complete bipartite graphs $K_{a'-b,a}$ and $K_{b,b'-a}$ by joining each vertex in the $a$ part of $K_{a'-b,a}$ to
each vertex in the $b$ part of $K_{b,b'-a}$.


\subsection{Definitions and known results for uniform hypergraphs}

We consider  {\em Berge cycles} in hypergraphs.

\begin{definition}
A {\bf Berge cycle of} length $c$ in a hypergraph is an alternating list of $c$ distinct vertices and $c$ distinct edges $C=v_1, e_1, v_2, \ldots,e_{c-1}, v_c, e_c, v_1$ such that $\{v_i, v_{i+1}\} \subseteq e_i$ for all $1\leq i \leq c$ (we always take indices of cycles of length $c$ modulo $c$).
We call vertices $v_1, \ldots, v_c$ {\bf the defining vertices} of $C$ and
the pairs $v_1 v_{2},v_2v_3, \ldots,v_c v_{1}$
{\bf the defining edges} of $C$. Given some edge $e_i \in E(C)$, we also call the defining edge $v_iv_{i+1}$ the {\bf the projection of} $e_i$. 
 We write
$V(C)=\{v_1, \ldots, v_c\}$, $E(C) = \{e_1, \ldots, e_c\}$. 
\end{definition}

Notation for  Berge paths is similar. In addition, a {\bf partial Berge path} is an alternating sequence of distinct  edges and vertices beginning with an edge and ending with a vertex $e_0, v_1, e_1, v_2, \ldots, e_k, v_{k+1}$ such that $v_1 \in e_0$ and for all $1\leq i \leq k$, $\{v_i, v_{i+1}\} \subseteq e_i$. 

The {\em circumference}, $c(H)$, of a (hyper)graph $H$ is the length of a longest (Berge) cycle in $H$.

A series of variations of Theorem~\ref{dirac} for Berge cycles in a number of classes of 
  $r$-graphs were obtained by
Bermond,  Germa,  Heydemann and Sotteau~\cite{BGHS}, Clemens, Ehrenm\"uller and Person~\cite{CEP},
Coulson and Perarnau~\cite{CP}, Ma, Hou, and Gao~\cite{MHG}, the present authors~\cite{KLM}, and Salia~\cite{SN} .

In particular, exact bounds for all values of $3\leq r<n$ are as follows.

\begin{thm}[Theorem 1.7 in~\cite{KLM}]\label{mainold2}
Let $t=t(n)=\lfloor \frac{n-1}{2} \rfloor$, and suppose $3 \leq  r <n$. Let $H$ be an $r$-graph. If
\hspace{1mm}
(a) $r\leq t$ and $\delta(H) \geq  {t \choose r-1} +1$ or 
\hspace{0.5mm}
(b) $r\geq n/2$ and $\delta(H) \geq r$, 
 \hspace{0.5mm}
 then $H$ contains a hamiltonian Berge cycle. 
\end{thm}


For an analog of Theorem~\ref{dirac2}, we define the connectivity of a hypergraph with the help of its {\em incidence bipartite graph}:

\begin{definition}Let $H$ be a hypergraph. The {\bf incidence  graph $I_H$ of $H$} is the bipartite graph with $V(I_H) = X \cup Y$ such that $X = V(H), Y = E(H)$ and for $x \in X, y\in Y$, $xy \in E(I_H)$ if and only if the vertex $x$ is contained in the edge $y$ in $H$.
\end{definition}

It is easy to see that if $H$ is an $r$-graph with minimum degree $\delta(H)$, then each $x \in X$ and each $y \in Y$ satisfy $d_{I_H}(x) \geq \delta(H), d_{I_H}(y) = r$. Furthermore, there is a natural bijection between the set of Berge cycles of length $c$ in $H$ and the set of cycles of length $2c$ in $I_H$:  such a Berge cycle $v_1, e_1, \ldots, v_c, e_c, v_1$ can also be viewed as a cycle in $I_H$ with the same sequence of vertices.

Using the notion of the incidence  graph, we  define connectivity in hypergraphs.

\begin{definition}
A hypergraph $H$ is {\bf $k$-connected} if its incidence  graph $I_H$ is a $k$-connected graph. 
\end{definition}

Theorem~\ref{jack} of Jackson applied to $I_H$ of a $2$-connected $r$-graph $H$ yields the following approximation of an analog of Theorem~\ref{dirac2}  for 
 $k\leq r-1$:

\begin{coro}\label{jackcor} Let $n, k, r$ be positive integers with $2 \leq k \leq r-1$. If $H$ is an  $n$-vertex $2$-connected $r$-graph $H$ with $\delta(H)\geq k+1$,
then  $c(H)\geq \min\{2k,n,|E(H)|\}$. 
\end{coro}

The following construction shows that the bound of Corollary~\ref{jackcor} is not far from exact.

\begin{construction}\label{cons1} 
 For $m\geq 2$, let 
$V(H_k)=A_1\cup \ldots \cup A_m\cup \{x,y\}$ where $A_i=\{a_{i,1},\ldots,a_{i,r-1}\}$ for $1\leq i\leq m$, and let
$E(H_k)=E_1\cup \ldots \cup E_m$ where for each $1\leq i\leq m$ and $1\leq j\leq k-1$, $E_i=\{e_{i,1},\ldots,e_{i,k-1}\}$ and
$e_{i,j}=(A_i-a_{i,j})\cup \{x,y\} $. By construction, $H_k$ is $2$-connected and $\delta(H_k)=k-2$.
Each Berge cycle in $H_k$ can contain edges from at most two $E_i$s, and $|E_i|=k-1$ for all $1\leq i\leq m$.
So, $c(H_k)=2k-2$.

\end{construction}

Very recently, the authors~\cite{KLM4} proved an exact analog of Theorem~\ref{dirac2} for $r$-graphs when $k\geq r+2$:

\begin{thm}\label{mainthm} Let $n, k, r$ be positive integers with $3\leq r \leq k-2\leq n-2 $. If $H$ is an  $n$-vertex $2$-connected $r$-graph with
\begin{equation}\label{main}
\delta(H) \geq {k-1 \choose r-1} + 1,
\end{equation}
then 
 $c(H)\geq \min\{2k,n\}$. 
\end{thm}

Observe that the minimum degree required to guarantee a Berge cycle of length at least $2k$  in a $2$-connected $r$-graph is roughly of the order $2^{r-1}/r$ times smaller than the sharp bound guaranteed in Theorem~\ref{mainold2}(b). 
The following constructions show the sharpness of Theorem~\ref{mainthm}.

\begin{construction}Let $q\geq 2$ be an integer and $4 \leq r+1 \leq k \leq n/2$. For $n = q(k-2)+ 2$, let $H_1 = H_1(k)$ be the $r$-graph with $V(H_1) = \{x,y\} \cup V_1 \cup V_2 \cup \ldots \cup V_q$ where  for all $1\leq i\leq q$,  $|V_i|=k-2$ and $V_i \cup \{x,y\}$ induces a clique.  Thus $c(H_1)\leq  2(k-2) + 2 = 2k-2$.
\end{construction}

\begin{construction}\label{constbip}Let $4 \leq r+1 \leq k \leq n/2$. Let $H_2 = H_2(k)$ be the $r$-graph with $V(H_2) = X \cup Y$ where $|X| = k-1$, $|Y| = n-(k-1)$, and $E(H_2)$ is the set of all hyperedges containing at most one vertex in $Y$. No Berge cycle can contain consecutive vertices in $Y$, so  $c(H_2)\leq 2k-2$. 
\end{construction}

One can check that both $H_1$ and $H_2$ have minimum degree ${k-1 \choose r-1}$ and $H_2$ is $(k-1)$-connected and is well defined for all $n\geq k$. 

It was also proved in~\cite{KLM4} that for $3\leq r<n$ and every  $n$-vertex $2$-connected $r$-graph  $H$,  $c(H)\geq \min\{4,n,|E(H)|\}$. 

\subsection{The main result and structure of the paper}

The ideas of~\cite{KLM4} were insufficient to prove exact results for $3\leq k\leq r+1$, because in this case the conditions providing a cycle of length at least $2k$ are weaker. For $3\leq k\leq r+1$, the restrictions on the minimum degree are linear in $k$ while for $k\geq r+2$ they are at least quadratic. Our main result is

\begin{thm}\label{mainthm3} Let $n, k, r$ be positive integers with $3\leq k \leq r+1\leq n $. If $H$ is an  $n$-vertex $2$-connected $r$-graph with
\begin{equation}\label{main2}
\delta(H) \geq k,
\end{equation}
then $c(H)\geq \min\{2k,n,|E(H)|\}$. 
\end{thm}

Note that under the conditions of the theorem it might occur that $|E(H)|<\min\{2k,n\}$, which could not happen in
Theorem~\ref{mainthm}.
The bound of the theorem coincides with the bound of 
 Theorem~\ref{dirac2} for 2-connected graphs when $|E(H)|\geq\min\{2k,n\}$. 
It
 is sharp when $k=3$ and when $k=r+1$. For $4\leq k\leq r$ the theorem improves the bound of 
Corollary~\ref{jackcor}, and Construction~\ref{cons1} shows that the bound either is exact or differs from the exact by  $1$.


The proof of our main result,
Theorem~\ref{mainthm3} is by contradiction. We consider a counterexample $H$ to the theorem and study its properties to show that such an example cannot exist.
In Section~\ref{setup!} we introduce  notation,
 define special substructures of $H$, 
so called {\em lollipops} and {\em disjoint cycle-path pairs},
and define when a structure in $H$ is {\em better} than another structure. In these terms, we explain the structure of the paper in more detail and state our main lemmas. In Section~\ref{simple}
 we derive some properties of ``good" 
lollipops and  disjoint cycle-path pairs. In  the remaining four sections we prove the four main lemmas stated in Section~\ref{setup!}.

\section{Notation and setup}\label{setup!}

For a hypergraph $H$, and a vertex $v \in V(H)$, \[N_{H}(v) = \{u \in V(H): \text{there exists } e \in E(H) \text{ such that } \{u,v\} \subset e\}\] is the {\em $H$-neighborhood} of $v$. 
%

When $G$ is a subhypergraph of a hypergraph $H$ and $u,v\in V(H)$,
we say that $u$ and $v$ are {\em $G$-neighbors} if there exists an edge $e \in E(G)$ containing both $u$ and $v$.

%
If $P$ is a (Berge) path and $a$ and $b$ are two elements of $P$, we use $P[a,b]$ to denote the unique segment of $P$ from $a$ to $b$. 

Let $r\geq 3$. Let $H$ be a 2-connected, $n$-vertex, $r$-uniform hypergraph satisfying~\eqref{main2}.
 Suppose 
 \begin{equation}\label{2kn}
 \mbox{\em $H$ does not contain a Berge cycle of length at least    $\min\{2k,n,|E(G)|\}$.}
 \end{equation}
 Recall that this minimum is not $n$ by Theorem~\ref{mainold2}.


A {\em lollipop $(C,P)$} is a pair where $C$ is a Berge cycle and $P$ is a Berge path or a partial Berge path that satisfies one of the following:

 \begin{enumerate}
 \item[---] $P$ is a Berge path starting with a vertex in $C$, $|V(C) \cap V(P)| = 1$, and $|E(C) \cap E(P)| = 0$. We call such a pair $(C,P)$ an {\bf ordinary lollipop} (or o-lollipop for short). See Fig.~\ref{pics} (left).
 
 \item[---] $P$ is a partial Berge path starting with an edge in $C$, $|V(C) \cap V(P)| = 0$, and $|E(C) \cap E(P)| = 1$. We call such a pair $(C,P)$ a {\bf partial lollipop} (or p-lollipop for short). See Fig.~\ref{pics} (middle).
 
 \end{enumerate}
 \begin{figure}
 \centering
 \begin{subfigure}[c]{0.62\textwidth}
 \includegraphics[width=\textwidth]{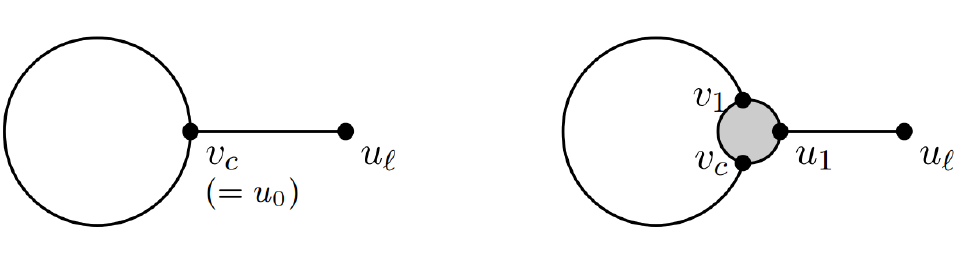}
 \end{subfigure} 
 \qquad \qquad
 \begin{subfigure}[c]{0.28\textwidth}
 \includegraphics[width=\textwidth]{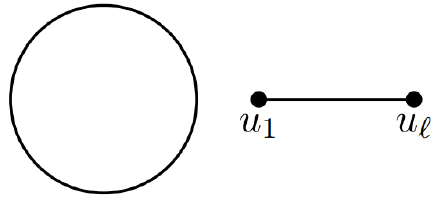}
 \end{subfigure}
 \caption{An o-lollipop, a p-lollipop, and a dcp-pair.}
 \label{pics}
 \end{figure}

 A lollipop  $(C,P)$  {\em is better} than a lollipop   $(C',P')$  if
  \begin{enumerate}
\item[(R1)]  $|V(C)|>|V(C')|$, or 
\item[(R2)] Rule (R1) does not distinguish $(C,P)$  from   $(C',P')$, and $|V(P)-V(C)|>|V(P')-V(C')|$; or

\item[(R3)] Rules (R1) and (R2) do not distinguish $(C,P)$  from   $(C',P')$, and the total number of vertices of $V(P)-V(C)$ contained in the  edges of $E(C)-E(P)$ counted with multiplicities is larger than the total number of vertices of $V(P')-V(C')$ contained in the  edges of $E(C')-E(P')$; or 
\item[(R4)] 
 Rules (R1)--(R3) do not distinguish $(C,P)$  from   $(C',P')$, and
 the number of edges in $E(P)-E(C)$ fully contained in $V(P)-V(C)$ is larger than the number of edges in $E(P')-E(C')$ fully contained in $V(P')-V(C')$. 
 \end{enumerate}

The criteria (R1)--(R4) define a partial ordering on the (finite) set of lollipops.
For $1\leq j\leq 4$, we will say that a lollipop is $j$-{\em good} if it is best among all lollipops according to the  rules (R1)--(Rj). For example,
a lollipop is $1$-good if the cycle in it is a longest cycle in $H$. Clearly if $i < j$ and a lollipop is $j$-good, then it is also $i$-good. We call a $4$-good lollipop a {\em best lollipop}.

Let  a lollipop $(C,P)$ be $j$-good for some $1\leq j\leq 4$. Say $C = v_1, e_1, \ldots, v_c, e_c, v_1$. If $(C,P)$ is a o-lollipop then let $P = u_0, f_0,u_1, \ldots, f_{\ell-1}, u_{\ell}$, where $u_0 = v_c$. If $(C,P)$ is a p-lollipop then let $P = f_0, u_1, f_1, \ldots, f_{\ell-1}, u_{\ell}$ where $f_0 = e_c$. With this notation, we have $|E(P)| = \ell$, $|V(P)| = \ell+1$ if $P$ is a Berge path, and $|V(P)|=\ell$ if $P$ is a partial Berge path. Assume $c < \min\{2k,|E(G)|,n\}$. 

A big part of the proof is to show that the case $\ell\geq k$ is impossible. After we prove this, we assume that $\ell\leq k-1$ and consider a somewhat weaker structure than lollipop, we call it a {\bf disjoint cycle-path pair} or {\em dcp-pair} for short. It is a pair $(C,P)$ of a cycle $C$ and a path $P$ without common edges and common defining vertices.
Similarly to lollipops, we say  that a dcp-pair $(C,P)$  {\em is better} than  another dcp-pair $(C',P')$  following the Rules (R1)--(R4) above. We say a dcp-pair $(C,P)$ is $j$-good for some $1 \leq j \leq 4$ if it is best among all dcp-pairs according to rules (R1)--(Rj).
A small simplification in the definition for dcp-pairs is that $V(P)-V(C)=V(P)$ and  $E(C)-E(P)=E(C)$.

With the help of dcp-pairs, we will show that there is 
a $2$-good lollipop $(C',P')$ with $|E(P)|\geq k$, thus obtaining the final contradiction.

So, the four big pieces of the proof are the following.

\begin{lemma}\label{nlem3}
Let $(C,P)$ be a best lollipop in $H$. If $|V(C)|=c$, $|E(P))|=\ell$ and
 $c < \min\{2k,n,|E(H)|\}$, 
then $\ell < k$.
\end{lemma}

\begin{lemma}\label{nlem2}
%
Let $(C,P)$ be a best lollipop in $H$ with $|V(C)|=c$, $|E(P))| < k$ and
 $c < \min\{2k,n,|E(H)|\}$. Then every $2$-good dcp-pair $(C', P')$ in $H$ has $|V(P')| < k$. 
\end{lemma}

\begin{lemma}\label{nlem0}
Let $(C,P)$ be a $3$-good dcp-pair in $H$. If $|V(C)|=c$, $|V(P))|=\ell$ and
 $c < \min\{2k,n,|E(H)|\}$, 
then $\ell \geq 2$.
\end{lemma}

\begin{lemma}\label{nlem1}
Let $(C,P)$ be a $4$-good dcp-pair in $H$. If $|V(C)|=c$, $|V(P))|=\ell$ and
 $c < \min\{2k,n,|E(H)|\}$, 
then $\ell \geq k$.
\end{lemma}

To prepare the proofs of Lemmas~\ref{nlem3}--\ref{nlem1}, in the next  section we 
describe some useful properties of good  lollipops and dcp-pairs. After that, in the next four sections we prove Lemmas~\ref{nlem3}--\ref{nlem1}.




\section{Simple properties of best lollipops}\label{simple}
In this subsection we consider $j$-good lollipops  and dcp-pairs $(C,P)$  and prove some basic claims to be used throughout the rest of the paper. We start from a simple but useful observation on graph cycles.

\begin{claim}\label{longQ}  Suppose $c < 2k$ and $C=v_1,e_1,v_2,\ldots,v_c,e_c,v_1$ is a graph cycle. Let $1\leq i<j\leq c$.
\begin{enumerate}
\item[(a)] The longer of the two subpaths of $C$ connecting 
$\{v_{i},v_{i+1}\}$ with $\{v_{j}, v_{j+1}\}$ and using neither 
 $e_i$ nor  $e_j$  has at least $\lceil c/2 \rceil $ vertices. 
 In particular, this path omits at most $k-1$ vertices in $C$. We call it a {\em long $e_i, e_j$-segment of $C$}.

\item[(b)] The longer of the two subpaths of $C$ connecting 
$\{v_{i},v_{i+1}\}$ with $v_{j}$ and not using 
 $e_i$  has at least $\lceil (c+1)/2 \rceil $ vertices. This path omits at most $k-1$ vertices in $C$. 
We call it a {\em long $e_i, v_j$-segment of $C$}.

\item[(c)] The longer of the two subpaths of $C$ connecting 
$v_{i}$ with $v_{j}$   has at least $\lceil (c+2)/2 \rceil $ vertices. In particular, this path omits at most $k-2$ vertices in $C$. We call it a {\em long $v_i, v_j$-segment of $C$}.
\end{enumerate}
\end{claim}

We will need the following useful notions.
Given an $r$-graph $H$ and a  lollipop or a dcp-pair $(C,P)$ in $H$, let 
 $H'=H'(C,P)$ denote the subhypergraph of $H$ with $V(H') = V(H)$ and $E(H') = E(H) - E(C) - E(P)$.


\begin{claim}\label{bigsmallcycle} Let $(C,P)$ be a $1$-good lollipop or a $1$-good dcp-pair in $H$. For each $1\leq i\leq c$ and $1\leq m\leq \ell$, if some edge $g\notin E(C)$ contains $\{u_{m},v_i\}$, then\\
(a) neither $e_{i-1}$ nor $e_i$ intersect $V(P)-u_0$, and\\
(b) if $(C,P)$ is $3$-good, then no edge in $H'$ intersects both  $V(P)-u_0$ and 
$\{v_{i-1},v_{i+1}\}$ (indices count modulo $c$).
In particular, the set $N_{H'}(V(P)-u_0)\cap V(C)$ does not contain two consecutive vertices of $C$.
\end{claim}
\begin{proof} Let $g \notin E(C)$ contain $\{u_m, v_i\}$.  If $g \in E(P)$, say $g = f_q$, then we may assume $u_m = u_{q+1}$. 

Suppose $ e_{i-1}$ contains $u_j$ for some $1\leq j\leq \ell$. If either $j \geq m$ or $g \neq f_m$, then we may replace the segment $v_{i-1}, e_{i-1}, v_i$ in $C$ with the path $v_{i-1}, e_{i-1}, u_{j},P[u_j,u_m],u_m, g, v_{i}$. Otherwise we replace the segment with the path $v_{i-1}, e_{i-1}, u_j, P[u_j, u_{m-1}], u_{m-1}, g, v_i$. We obtain a longer cycle, contradicting the choice of $C$. The case with $u_{j} \in e_i$ is symmetric. This proves (a).

Suppose now some $e\in E(H')$ contains $\{u_{j},v_{i-1}\}$ for some $1\leq j\leq \ell$ (the case when $e\supset \{u_{j},v_{i+1}\}$ is symmetric).
If $e\neq g$, then as in the proof of (a) we may replace the segment $v_{i-1}, e_{i-1}, v_i$ in $C$ with $v_{i-1}, e, u_{j},P[u_j,u_m],u_m, g, v_{i}$ or $v_{i-1}, e, u_j, P[u_j, u_{m-1}], u_{m-1}, g$
 to get a longer cycle.

If $e = g$, then by (a), $ e_{i-1}\cap (V(P)-u_0)=\emptyset$. Note that by the case $g=e \in E(H')$. Let $C'$ be obtained from $C$ by
 replacing the edge $e_{i-1}$ with $g$ and let $P'=P$. If $i\neq 1$ or $(C,P)$ is a dcp-pair, 
 then  $(C',P')$ is better than  $(C,P)$ by Rule (R3).  If $i= 1$ and $(C,P)$ is an o-lollipop or a p-lollipop, 
 then the cycle obtained from $C$ by replacing subpath $v_c,e_c,v_1$ with the path $v_c,f_1,u_1,\ldots,u_j,g,v_1$ is longer than $C$.
\end{proof}


\begin{claim}\label{shortpath}  Let $(C,P)$ be a $1$-good lollipop or a $1$-good dcp-pair in $H$. 
For $1\leq q \leq \ell-1$ and $1 \leq i,j \leq c$, the following hold:

(i) If $u_q \in e_i$ and $u_{\ell} \in e_j$ then $j=i$ or $|j-i| \geq \ell - q + 1$.

(ii) If $\{v_i, u_q\} \subset e$ for some edge $e \in (E(H') \cup \{f_0\})-E(C)$ and if $u_{\ell} \in e_j$, then either $j > i$ and $j-i \geq \ell - q+1$, or $i > j$ and $i-j \geq \ell-q+2$.

(iii) If there exist distinct edges $e, f \in (E(H') \cup \{f_0\})-E(C)$ such that $\{v_i, u_q\} \subset e$ and $\{v_j, u_{\ell}\} \subset f$, then $j=i$ or $|j-i| \geq \ell - q + 2$. 
%

\end{claim}

\begin{proof} We will prove (i). If $j\neq i$, then we can replace the segment of $C$ from $e_i$ to $e_j$ containing $|j-i|$ vertices with the path $e_j, u_{\ell}, P[u_{\ell}, u_q], u_q, e_i$ which contains $\ell-q + 1$ vertices. The new cycle  cannot be longer than $C$, thus (ii) holds. The proofs for (ii) and (iii) are similar so we omit them.
\end{proof}

%


\begin{claim}\label{allneighbors}
Suppose $(C,P)$ is a $2$-good lollipop or $2$-good dcp-pair.
 Then all $H'$-neighbors of $u_{\ell}$ are in $V(C) \cup V(P)$.  Moreover,

(a) if $(C,P)$ is an $o$-lollipop, then $u_{\ell }$ has no $H'$-neighbors in $\{v_1, v_2, \ldots, v_{\ell}\} \cup \{v_{c-1}, v_{c-2}, \ldots, v_{c-\ell}\}$, and $u_{\ell }$ is not in any edge in the set $\{e_1, e_2, \ldots, e_{\ell-1}\} \cup \{e_c, e_{c-1}, \ldots, e_{c-\ell}\}$, 

(b) if $(C,P)$ is a $p$-lollipop, then $u_{\ell}$ has no ${H}'$-neighbors in $\{v_1, v_2, \ldots, v_\ell\} \cup \{v_c, v_{c-1}, \ldots, v_{c-\ell+1}\}$, and $u_{\ell}$ is not in any edge in the set $\{e_1, \ldots, e_{\ell-1}\} \cup \{e_{c-1}, \ldots, e_{c-(\ell-1)}\}$.
\end{claim}

\begin{proof} Let $g \in E(H')$ contain $u_{\ell}$. Suppose first there is a vertex $y \in V(H) - (V(C) \cup V(P))$ such that $y \in g$. Let $P'$ be the path  obtained 
from $P$ by adding  edge $g$ and vertex $y$ to the end of $P$. Then $(C, P')$ is a lollipop with $|V(P')| > |V(P)|$, a contradiction.

Part (a) follows from Claim~\ref{shortpath}(ii,iii) for $q=1$ since $f_0 \notin E(C)$ and contains $\{u_1,v_c\}$. Part (b) follows from Claim~\ref{shortpath}(i,ii) for $q=1$ since $f_0 = e_c$ contains $\{u_1, v_c, v_1\}$.
\end{proof}





\begin{claim}\label{bestp}  Let $(C,P)$ is a $j$-good lollipop or $j$-good dcp-pair for some $1\leq j\leq 4$.


(A) If $u_{\ell}\in f_m$ for some $1\leq m\leq \ell-2$ and $P_{m+1}$ is obtained from $P$ by replacing the subpath 
$u_m,f_m,u_{m+1},\ldots,u_{\ell}$ with the subpath $u_m,f_m,u_{\ell},f_{\ell-1},u_{\ell-1},\ldots,u_{m+1}$, then 
$(C,P_{m+1})$ also is a $j$-good lollipop or dcp-pair.

(B) If some edge $g\in E(H')$ contains $V(P)-V(C)$ or is contained in $V(P)-V(C)$  and contains $\{u_{\ell},u_m\}$  for some $1\leq m\leq \ell-2$, and if $P'_{m+1}$ is obtained from $P$ by replacing the subpath 
$u_m,f_m,u_{m+1},\ldots,u_{\ell}$ with the subpath $u_m,g,u_{\ell},f_{\ell-1},u_{\ell-1},\ldots,u_{m+1}$, then $(C,P'_{m+1})$ also is a $j$-good lollipop or dcp-pair.
\end{claim}

\begin{proof} Let us check the definition of a  $j$-good lollipop or dcp-pair.
 Part (A) holds because  the vertex set and edge set of $P_{m+1}$ are the same as those of $P$.
 
 In Part (B),  $V(P'_{m+1})-V(C)=V(P)-V(C)$, and $E(P'_{m+1})$ is obtained from  $E(P)$ by deleting $f_m$ and adding $g$. But since $g$
 contains $V(P)-V(C)$ or is contained in $V(P)-V(C)$, $(C,P)$ cannot be better than $(C,P'_{m+1})$.
\end{proof}

\section{Proof of Lemma~\ref{nlem3}: paths in lollipops must be short}

Let $(C,P)$ be a best lollipop with $C = v_1, e_1, v_2, \ldots, v_c, e_c, v_1$. Say $P = u_0, f_0, u_1, f_1, \ldots, u_\ell$ with $u_0= v_c$ if $(C,P)$ is an o-lollipop, and $P = f_0, u_1, f_1, \ldots, u_\ell$ with $f_0 = e_c$ if it is a p-lollipop. 

In this section we prove that if $P$ is long, then we can find a longer cycle than $C$. 
For this we will use a modification of a lemma from Dirac's original proof of Theorem~\ref{dirac} in~\cite{D}.
Let $Q$ and $Q'$ be two (graph) paths in a graph $G$. We say $Q$ and $Q'$ are {\em aligned} if for every $x, y \in V(Q) \cap V(Q')$, $x$ appears before $y$ in $Q$ if and only if $x$ appears before $y$ in $Q'$. 

\begin{lemma}[Lemma 5 in \cite{KLM2}]\label{diraclemma}Let $Q$ be an $x,y$-path in a $2$-connected graph $G$, and let $z \in V(G)-\{y\}$. Then there exists an $x,z$-path $P_1$ and an $x,y$-path $P_2$ such that (a) $V(P_1) \cap V(P_2) = \{x\}$, and (b) each of $P_1$ and $P_2$ is aligned with $Q$. 

\end{lemma}

We call a vertex $x\in V(H)$ {\em eligible} if there exists a best lollipop $(C',P')$ where $x$ is the end vertex of $P'$ that is not contained in $C'$.  In particular, by Claim~\ref{bestp} in our best lollipop $(C,P)$, for any $f_m \in E(P)$ containing $u_{\ell}$, $u_{m+1}$ is eligible by considering the best lollipop $(C, P_{m+1})$, where $P_{m+1}$ is defined as in  Claim~\ref{bestp}.

Set $u_0 = v_c$ if $(C,P)$ is a p-lollipop. Recall if $(C,P)$ is a p-lollipop, then $E(C) \cap E(P) = e_c = f_0$.  Define \[S_1 = N_{H'}(u_{\ell}) \cap V(P) \text{ and } S_2 = \{u_m: 0\leq m \leq \ell-1, u_{\ell} \in f_m, u_m \notin S_1\}.\] Observe that $(S_1 \cup S_2) \cap V(C) \subseteq \{u_0\}$.

We will prove the following for $(C,P)$ and eligible vertex $u_{\ell}$, but all proofs will work for {\em any} best lollipop and corresponding eligible vertex. 
\begin{lemma}\label{maindirac}If $\delta(H) \geq k$ and $|E(P)| = \ell \geq k$, then (A) $|S_1 \cup S_2| \leq k-1$ and (B) $|S_1| \leq k-2$. 
\end{lemma}



\begin{proof}
Throughout this proof we will use $i_1$ to denote the smallest index such that $u_{i_1} \in S_1 \cup S_2$ and $j_1$ to denote the smallest index such that $u_{j_1} \in S_1$. 

Since $\ell \geq k$, $u_\ell$ has no $H'$-neighbors in $V(C) - V(P)$ by Claim~\ref{allneighbors}. That is, $S_1 =N_{H'}(u_{\ell})$. 

\begin{claim}\label{cXY} (A) If $|S_1 \cup S_2| \geq k$ and some edge  $g \in E(C) \cup E(H')$ intersects both $X$ and $\{u_{i_1+1}, \ldots, u_{\ell}\}$, then $(C,P)$ is a p-lollipop and $g = e_c (=f_0)$.

(B) If $|S_1| \geq k-1$, then no edge in $E(H')$ intersects both $X$ and $\{u_{j_1+1}, \ldots, u_\ell\}$. 
\end{claim}

\begin{proof}
Suppose 
   $g \in E(C) \cup E(H')$ is an edge intersecting $X$ and $\{u_{i_1+1}, \ldots, u_\ell\}$, say $\{v_a, u_b\} \in g$ where $b \geq i_1 + 1$, and without loss of generality $g = e_{a-1}$ if $g \in E(C)$. By symmetry, if $g \in E(H')$ then we may assume $a \leq \lfloor c/2 \rfloor$ when $(C,P)$ is an $o$-lollipop, and $a \leq \lceil c/2 \rceil$ when $(C,P)$ is a p-lollipop, and if $g = e_{a-1}$,  then we may assume $a-1 \leq \lfloor c/2 \rfloor$ or $a-1=c$ (in this case, $(C,P)$ is an o-lollipop, otherwise we have proven the claim).

 By Claim~\ref{allneighbors}, $u_{\ell} \notin g$, so $b \leq \ell-1$. Let $j$ be the largest index smaller than $b$ such that $u_j \in S_1 \cup S_2$. Since $i_1 < b$, such an $j$ exists.
Let $h$ be an edge in $E(H') \cup f_{j}$ that contains $\{u_j,u_{\ell}\}$. Set \[C' = v_c, e_{c-1}, v_{c-1}, \ldots, v_a, g, u_b, f_b, \ldots, u_{\ell}, h, u_j, f_{j-1}, \ldots, u_0 (=v_c).\] 

The cycle $C'$ contains all vertices in $S_1 \cup S_2 \cup \{u_{\ell}\}$, and among these vertices, only $u_0$ may belong to $C$. Moreover, $C'$ contains at least $c-(a-1) \geq c-(\lfloor c/2 \rfloor) \geq c-(k-1)$ vertices in $C$. 
Therefore  
$$|C'| \geq c-(a-1) + |S_1 \cup S_2 \cup \{u_{\ell}\}| - |(S_1 \cup S_2) \cap V(C)| \geq c-(k-1) + k + 1 -1 > c=|C|,$$ contradicting the choice of $C$. This proves (A). 

The proof for (B) is similar but we replace $S_1 \cup S_2$ with $S_1$ and $i_1$ with $j_1$. Define the cycle $C'$ as before, and observe that if $(C,P)$ is an o-lollipop, then $C'$ contains at least $c-(a-1) \geq c-(\lfloor c/2 \rfloor -1) \geq c-(k-2)$ vertices in $C$, and if $(C,P)$ is a p-lollipop, then $C'$ contains at least $c-(k-1)$ vertices in $C$, but $S_1 \cap V(C) = \emptyset$. In either case, $|C'| \geq c-(k-1) + (k-1) + 1 > |C|$.
\end{proof}


Recall that $i_1$ is the smallest index with $u_{i_1} \in S_1 \cup S_2$. If $u_{i_1} \in S_1$ set $\beta = i_1$; otherwise set $\beta ={i_1}+1$.
%

\begin{claim}\label{fm1'} Suppose $|S_1\cup S_2|\geq k$. 
Then 

(a) there exists an edge $f_j$ with $j \geq \beta$ that intersects $V(C)$,

(b) $\{u_{j}, \ldots, u_{\ell-1}\} \in S_1 \cup S_2$, 
 
(c) $(C,P)$ is a p-lollipop, and

(d) $u_{\ell}\in f_0$ and $u_{\ell} \in f_j$ for every $f_j$ satisfying (a).  
\end{claim}

\begin{proof}
 Consider the 2-connected incidence graph $I_H$ of $H$ and the (graph) path
  \[P' = v_1, e_1, v_2, \ldots, v_c, f_0,u_1, \ldots, f_{\ell}, u_{\ell}\] in $I_H$. 
  We apply Lemma~\ref{diraclemma} to $P'$ with $z = u_\beta$ to obtain two internally disjoint (graph) paths $P_1$ and $P_2$ such that $P_1$ is a $v_1, z$-path, $P_2$ is a $v_1, u_\ell$-path, and each $P_i$ is aligned with $P'$.

We modify $P_i$ as follows: if $P_i = a_1, a_2, \ldots, a_{j_i}$, let $q_i$ be the last index such that $a_{q_i} \in X':= \{v_1, e_1, \ldots, v_c, e_c\}$ and let $p_i$ be the first index such that $a_{p_i} \in Y':=\{u_{\beta}, f_{\beta}, u_{\beta+1}, \ldots, f_{\ell}, u_{\ell}\}$. 

If $a_{p_i} =u_s$ for some $s$, then set $P_i'= P_i[a_{q_i}, a_{p_i}]$. If $a_{p_i} = f_s$ for some $s$, then set $P_i' = P_i[a_{q_i}, a_{p_i}], u_{s+1}$. 

Observe that $P_1'$ and $P_2'$ are either Berge paths or partial Berge paths in $H$. Moreover, $P_1'$ ends with vertex $z = u_{\beta}$
and contains no other elements of $Y'$ since it is aligned  with $P'$. 

If both $P_1'$ and $P_2'$  begin with $v_1$,  then some $P_i'$ avoids $f_0$ and first intersects the set $\{u_1, f_1, \ldots, u_{\ell}\}$ at some element $a_j$. Then replacing the segment $v_1, e_c, v_c$ in $C$ with the longer segment $v_1, P_i'[v_1, a_j], a_j, P[a_j, f_0], f_0, v_c$ yields a cycle in $H$ that is longer than $C$, a contradiction. Therefore we may assume that $P_1'$ and $P_2'$ are vertex-disjoint and edge-disjoint in $H$. 
 Let $u_g$ be the last vertex of $P'_2$. We have $g > \beta$.

If $f_{g-1} \notin E(P'_2)$, or if $u_{g-1} \in S_1$, let $g'$ be the largest index less than $g$ such that $u_{g'} \in S_1 \cup S_2$. Otherwise if $f_{g-1} \in E(P_2')$ and $u_{g-1} \notin S_1$,  let 
 $g'$ be the largest index less $g-1$ such that $u_{g'}\in S_1\cup S_2$.
If $u_{g'} \in S_1$, let $h \in E(H')$ contain $\{u_{g'}, u_{\ell}\}$; otherwise, let $h = f_{g'}$.  We claim

\begin{equation}\label{hdisjoint}
h \notin E(P_1' \cup P_2').
\end{equation} 
Suppose not. If $h \in E(H')$, then $h$ contains a vertex outside of $\{u_{\beta}, \ldots, u_{\ell}\}$. But this violates either Claim~\ref{cXY} (B) or the definition of $S_1 = N_{H'}(u_{\ell}) \subseteq \{u_{\beta}, \ldots, u_{\ell}\}$. If $h = f_{g'}$, then $f_{g'} \in E(P'_2)$. By construction, $u_g=u_{g'+1}$, but in this case we chose $g'$ such that $g' < g-1$, a contradiction. This proves~\eqref{hdisjoint}.
 
 By Claim~\ref{longQ}, there exists a long $a_{q_1}, a_{q_2}$-segment $Q$ of $C$ such that $|V(Q)| \geq c-(k-1)$ with equality only if at least one of $a_{q_1}$ or $a_{q_2}$ is an edge of $C$. 
  
By~\eqref{hdisjoint} we may define the cycle
\[C'=a_{q_1}, Q, a_{q_2}, P'_2, u_g, f_{g}, \ldots, u_{\ell}, h, u_{g'}, f_{g'-1}, \ldots, u_\beta, P_1', a_{q_1}.\]

Observe that $C'$ contains the set $U = \{u_\beta, \ldots, u_{g'}\} \cup \{u_g, \ldots, u_{\ell}\}$. If $f_{g-1} \notin E(P_2')$ or $g'=g-1$, then $S_1 \cup S_2 \cup \{u_{\ell}\} - \{u_{i_1}\} \subseteq U - V(C)$. Otherwise, $S_1 \cup S_2 \cup \{u_{\ell}\} - \{u_{g-1}, u_{i_1}\} \subseteq U-V(C)$. Therefore $|U-V(C)| \geq |S_1 \cup S_2| +1 - 2 \geq k-1$ with equality only if $f_{g-1} \in E(P_2')$ and $u_{g-1} \in S_2$. 
We have
\begin{equation}\label{C'}
|C'| \geq |V(Q)| + |V(P_2' \cup P_1') - V(C) - U| + |U-V(C)|.
\end{equation}

If $P_1'$ or $P_2'$ contains a vertex outside of $V(C)  \cup U$, then $|C'| \geq c-(k-1) + 1 + (k - 1) > |C|$, contradicting the choice of $C$. Thus by construction, we may assume that 
\begin{equation}\label{1edge}
\hbox{$P_1'$ and $P_2'$ each contain at most one edge.}
\end{equation}
Similarly, if $|V(Q)| \geq c-(k-1)+1$ or if $|U-V(C)| \geq k$ then $|C'| \geq c-(k-1) + 0 + (k-1) + 1 > |C|$, a contradiction. So $f_{g-1} \in E(P_2')$ and  $u_{g-1} \in S_2$, proving (a). By~\eqref{1edge} and Claim~\ref{cXY}, $P_2' = v_i, f_{g-1}, u_g$ for some $v_i \in V(C)$. Also, we must have $U -V(C)= S_1 \cup S_2 \cup\{u_{\ell}\} - \{u_{g-1},u_{i_1}\}$ which has size exactly $k-1$. In particular, $\{u_g, u_{g+1}, \ldots, u_{\ell-1}\} \subseteq S_1 \cup S_2$, so (b) holds. 

In order to have $|V(Q)| = c-(k-1)$ by Claim~\ref{longQ}, $P_1'$ must begin with an edge of $C$. Then by~\eqref{1edge} and Claim~\ref{cXY} (A), we must have $(C,P)$ is a p-lollipop and $P_1'=e_c, u_\beta$. Therefore (c) holds. We have shown that $u_{\ell} \in f_{g-1}$. If $f_s$ is another edge with $s \geq \beta$ that intersects $V(C)$, say at vertex $v_{s'}$, then we may substitute $P_2' = v_{s'}, f_s, u_{s+1}$ (which is disjoint from $P_1'$) and symmetrically obtain that $u_{\ell} \in f_s$ as well. If $u_{\ell} \notin f_0$ (so $i_1 \geq 1$ by Claim~\ref{allneighbors}), then the cycle \[C'' = a_{q_1}, Q, a_{q_2}, f_{g-1}, u_g, \ldots, u_{\ell}, h, u_{g'}, \ldots,u_1, f_0 (=e_c=a_{q_1})\] contains all of $S_1 \cup S_2 \cup \{u_{\ell}\} - \{u_{g-1}\}$, and this set is disjoint from $V(C)$. Therefore $|C''| \geq c-(k-1) + k > |C|$. This proves (d).
%
%
%
\end{proof}

\begin{claim}\label{1int}$f_{\ell-1} \cap V(C) \subseteq  \{v_{\lceil c/2 \rceil}\}$. 
\end{claim}
\begin{proof}
If $f_{\ell-1}$ contains a vertex $v_j \in V(C) - \{v_{\lceil c/2 \rceil}\}$, without loss of generality we may assume $j \geq \lceil c/2 \rceil+1$. Then $C' = v_1, e_1, \ldots, v_j,f_{\ell-1}, u_{\ell-1}, f_{\ell-2}, \ldots, u_1, f_0, v_1$ has length at least $\lceil c/2 \rceil +1 + |V(P)-V(C)|-1 \geq c-(k-2) + \ell-1 > c$, a contradiction. 
\end{proof}

\begin{claim}\label{eligible1}If $|S_1 \cup S_2| \geq k$ and $u_i \in S_1 \cup S_2$, then $u_{i+1}$ is eligible.
\end{claim}
\begin{proof}By Claim~\ref{fm1'}, $(C,P)$ is a p-lollipop.
If $u_i \in S_1$ then let $h \in E(H')$ contain $\{u_i, u_{\ell}\}$. By Claim~\ref{allneighbors}, $h \subseteq V(P)-V(C)$. 
The result follows from Claim~\ref{bestp}.
\end{proof}

\begin{claim}\label{fm1''}  $|S_1\cup S_2|\leq k-1$. 
\end{claim}

\begin{proof} Suppose $|S_1\cup S_2|\geq k$. By Claim~\ref{fm1'}, $(C,P)$ is a p-lollipop and $u_{\ell} \in f_0$. Let $f_j$ be an edge with $j \geq \beta\geq 1$ intersecting $V(C)$. By Claims~\ref{fm1'} and~\ref{eligible1}, $u_{j+1}, u_{j+2}, \ldots, u_{\ell}$ are also eligible vertices. Thus applying Claim~\ref{fm1'} to these vertices and their corresponding best lollipops $(C, P_{i+1}')$ (where $P_{i+1}'$ is defined as in Claim~\ref{bestp}) imply that $\{u_{j+1}, \ldots, u_\ell, u_{\ell}\} \subseteq f_0$. 

Symmetrically consider $P' = f_0,u_{\ell},f_{\ell-1},u_{\ell-1},\ldots,u_1$ and observe that $(C, P')$ is a best lollipop and $P'$ has first edge $f_0$. Applying Claims~\ref{fm1'} and~\ref{eligible1} to $(C, P')$ and the eligible vertex $u_1$, we obtain that $u_1 \in f_0$ and vertices $u_{2}, u_{3}, \ldots, u_{j}$ are eligible and therefore contained in $f_0$. Thus $V(P) \subseteq f_0$ and so $r= |f_0| \geq |V(P) \cup \{v_1,v_c\}|$. 
%
%
%

No edge in $C$ may contain $u_{\ell}$ by Claim~\ref{allneighbors}. If a vertex $w\in f_{\ell-1}$ is outside of $V(C)\cup V(P)$, then we can replace $u_{\ell}$ with $w$ in $P$ and get another best lollipop. So again by Claim~\ref{fm1'}, $w\in f_0$. By Claim~\ref{1int}, $f_{\ell-1}$ contains at most one vertex in $V(C)$. Then $f_0$ contains all of $f_{\ell-1}-V(C)$ as well as two vertices $v_1$ and $v_c$ in $C$. This contradicts that $H$ is $r$-uniform.
\end{proof}

We are now ready to prove Lemma~\ref{maindirac}. It remains to prove (B), so 
suppose towards contradiction that $|S_1| \geq k-1$. By Claim~\ref{fm1''}, $|S_1|=|S_1 \cup S_2| = k-1$ and $S_1 = S_1 \cup S_2$. As in the Proof of Claim~\ref{fm1''}, let $P' = v_1, e_1, v_2, \ldots, v_c, f_0,u_1, \ldots, f_{\ell}, u_{\ell}$, and apply Lemma~\ref{diraclemma} to the incidence graph $I_H$ with $P'$ and $z = u_{j_1}$. We obtain two aligned with $P'$ paths and modify them to get Berge paths or partial Berge paths $P_1'$ and $P_2'$ such that $P_1'$ starts in $C$ and ends in $u_{j_1}$, and $P_2'$ starts in $C$ and ends in a vertex $u_g$ with $g > j_1$. Also, if $P_2'$ contains some edge $f_i$ with $i \geq j_1$, then we may assume $g = i+1$. 

Let $a_{q_1}$ and $a_{q_2}$ be the first elements in $P_1'$ and $P_2'$ respectively. Let $Q$ be a long $a_{q_1}, a_{q_2}$-segment in $C$. As before, $|V(Q)| \geq c-(k-1)$ with equality only if at least one of $a_{q_1}$ and $a_{q_2}$ is an edge.

Let $g'$ be the largest index less than $g$ such that $u_{g'} \in S_1$ and let $h \in E(H')$ contain $\{u_{g'}, u_{\ell}\}$. Define \[C' = a_{q_1}, Q, a_{q_2}, P_2', u_g, f_g, \ldots, u_{\ell}, h, u_{g'}, \ldots, u_{j_1}, P_1', a_{q_1}.\]

The cycle $C'$ contains the set $U = \{u_{j_1}, \ldots, u_{g'}\} \cup \{u_g, \ldots, u_{\ell}\}$ which contains $S_1 \cup \{u_{\ell}\}$ and intersects $V(C)$ in at most one vertex, $u_0$. Therefore  
\begin{equation}\label{eq1}\hbox{$|U-V(C)| \geq k-1 + 1 - 1=k-1$ with equality only if $S_1 \cup \{u_\ell\} = U$ and $u_0 \in S_1$.} \end{equation}

If $P_1'$ or $P_2'$ contain any internal vertices, if $|U-V(C)| \geq k$, or if $|V(Q)| \geq c-(k-1) + 1$, then $|C'| \geq c-(k-1) + (k-1) + 1 > |C|$, a contradiction. Thus $P_1'$ and $P_2'$ contain at most one edge, $|U-V(C)| = k-1$, and $|V(Q)| = c-(k-1)$. Then~\eqref{eq1} holds. In particular, $u_0 \in S_1$ implies $(C,P)$ is an o-lollipop and $P_1' = u_0$. By Claim~\ref{longQ} and the fact $|V(Q)| = c-(k-1)$, $P_2' = e_j, u_g$ for some $e_j \in E(C)$. Moreover we must have $c=2k-1$ and $e_j = e_{k-1}$, otherwise $|V(Q)| \geq c-(k-1)+ 1$.  
By Claim~\ref{allneighbors} and the fact $\ell \geq k$, $g \leq \ell-1$ and therefore $u_{\ell-1} \in S_1$.

\begin{claim}\label{insideP}For all $u_s \in S_1$, then $f_s \subseteq V(P)$. \end{claim}

\begin{proof}Let $h' \in E(H')$ contain $\{u_s, u_{\ell}\}$ and set $P'_{s+1} = u_0, f_0, \ldots, u_s, h', u_\ell, f_{\ell-1}, \ldots, u_{s+1}$.  The lollipop $(C, P_{s+1}')$ is a 3-good lollipop that omits edge $f_s$ which is incident to the end vertex $u_{s+1}$ of $P_{s+1}$. By Claim~\ref{allneighbors} applied to $(C,P_{s+1}')$, $f_s \subseteq V(P_{s+1}') = V(P)$.
\end{proof}

\begin{claim}\label{insideP2}For every $s \geq g$, $f_s \subseteq S_1 \cup \{u_{\ell}\}$.\end{claim}
\begin{proof} By Claim~\ref{insideP} and~\eqref{eq1}, $f_s \subseteq V(P)$. Suppose $f_s$ contains a vertex $u_{s'} \in V(P) - U= \{u_{g'+1}, \ldots, u_{g-1}\}$.

Let $h' \in E(H')$ be an edge in $E(H')$ containing $\{u_s, u_{\ell}\}$. Then
\[C'' = v_c, e_c, v_1, \ldots, v_{k-1}, e_{k-1}, u_g, \ldots, u_s, h', u_{\ell}, f_{\ell-1}, \ldots, u_{s+1}, f_s, u_{s'}, \ldots, u_0 (=v_c)\] contains $U \cup \{u_{\ell}\} \cup \{u_{g'+1}, \ldots, u_{s'}\}$ and therefore is longer than $C'$ (and $C$). 
\end{proof} 

Recall that $S_1 = S_1 \cup S_2$, and so if $u_{\ell} \in f_s$, then $u_s \in S_1$. If every $f_s$ containing $u_{\ell}$ is a subset of $S_1 \cup \{u_{\ell}\}$, then $d_{H}(u_{\ell}) \leq {|S_1| \choose r-1} ={k-1 \choose r-1} < k=\delta(H)$, a contradiction. So by the previous Claim, $u_{\ell}$ must be contained in some $f_s$ with $s \leq g'$. If there exists  an edge $h' \in E(H')$ containing $\{u_s, u_{g+1}\}$, then the cycle 
\[C'' = v_c, e_c, v_1, \ldots, v_{k-1}, e_{k-1}, u_g, f_{g-1}, \ldots, u_{s+1}, f_s, u_{\ell}, f_{\ell-1}, \ldots, u_{g+1}, h', u_s, \ldots, u_0\]
contains all of $V(P)$ and is longer than $C$, a contradiction. Otherwise every edge in $E(H')$ that intersects $u_{\ell}$ contains $r-1 \geq k-2$ vertices in the $(k-1)$-set $S_1$ but does not contain both $u_s$ and $u_{g+1}$. This implies $r = k-1$ and since $|S_1 \cup \{u_\ell\}| = k > r$, $u_{\ell}$ is contained in at least two edges in $H'$. But then there can  be only  two such edges, one that avoids $u_s$ and one that avoids $u_{g+1}$. Let $h''$ be the one avoiding $u_s$. By Claim~\ref{insideP2}, $f_{\ell-1}$ also contains $k-2$ vertices in $S_1$ and therefore must contain both $u_s$ and $u_{g+1}$. Then we replace $h'$  in $C''$ with $f_{\ell-1}$ and $f_{\ell-1}$ with $h''$ to obtain another longer cycle.
\end{proof}

As a corollary we have the following.

\begin{cor}\label{smalldeg}
Suppose $(C',P')$ is a best lollipop with $C' = v_1', e_1', v_2', \ldots, v_c', e_c', v_1'$ and $P' = u_1', f_1', \ldots, f_{\ell}', u_{\ell}'$. If $\ell = |V(P')| \geq k$, then
\begin{enumerate} 
\item[(i)]$d_{P'}(u_{\ell}') \leq k-1$,
\item[(ii)]$d_{H-C' - P'}(u_{\ell}') \leq 1$ with equality only if $r = k-1$,  and
\item[(iii)]$d_{C'-P'}(u_{\ell}') = 0$.
\end{enumerate}
Moreover if $d_{H-C'-P'}(u_{\ell}') = 1$ and $d_{P'-C'}(u_{\ell}') = k-1$, then $u_{\ell}' \in f_i'$ for every $u_i' \in N_{H - C'-P'}(u_{\ell}')$.
\end{cor}
\begin{proof}
 
Part (i) follows from Lemma~\ref{maindirac} (A). By Lemma~\ref{maindirac} (B), $|S_1| \leq k-2 \leq r-1$. Thus if $N_{H-C'-P'}(u_{\ell})$ is nonempty, then $|N_{H-C'-P'}(u_{\ell})| = r-1$. That is, $d_{H-C'-P'}(u_{\ell}) \leq 1$ and we obtain Part (ii).
Part (iii) follows from Claim~\ref{allneighbors}.

The ``moreover" part comes from the fact that that $\{u_i': u_{\ell}' \in f_i'\} \subseteq S_1 \cup S_2$. But if $d_{P'}(u_{\ell}') = k-1 = |S_1 \cup S_2|$ then $S_1 \subseteq \{u_i': u_{\ell}' \in f_i'\}$.
\end{proof}
Finally we are ready to show that the paths in lollipops cannot be long. Recall the statement of Lemma~\ref{nlem3}.

{\bf Lemma~\ref{nlem3}.} {\em Let $(C,P)$ be a best lollipop in $H$. If $|V(C)|=c$, $|E(P))|=\ell$ and
 $c < \min\{2k,n,|E(H)|\}$, 
then $\ell < k$.}

\begin{proof}[Proof of Lemma~\ref{nlem3}.]
Suppose  that $\ell \geq k$.
If $r \geq k$, then by Corollary~\ref{smalldeg}, \[d_H(u_{\ell}) \leq d_{H'}(u_{\ell}) + d_{C-P}(u_{\ell}) + d_P(u_{\ell}) \leq 0 + 0 + k-1 < \delta(H).\] So  assume $r=k-1$. Let $u_{i+1}$ be an eligible vertex. By Corollary~\ref{smalldeg} applied to $(C,P_{i+1})$ (note that $E(P) = E(P_{i+1})$), $d_{C-P}(u_{i+1}) = 0$, and therefore $d_P(u_{i+1}) = k-1$ and $d_{H'}(u_{i+1}) = 1$. Let $e \in E(H')$ contain $u_{\ell}$. Say $e=\{u_{j_1}, \ldots, u_{j_r}\}$ where $j_1 < \ldots < j_r = \ell$. 
Also by Corollary~\ref{smalldeg}, $u_{\ell}$ belongs to each of the edges $f_{j_1}, \ldots, f_{j_r}$.

If possible,  we choose a best lollipop $(C,P)$ such that $j_1 > 0$.

{\bf Case 1}: $j_1 > 0$. Consider the incidence graph $I_H$.
Set $X = C ( = V(C) \cup E(C))$ and $Y = P[u_{j_1}, u_{\ell}]$. 
As in the proof of Lemma~\ref{maindirac} we apply the modification of Dirac's Lemma, Lemma~\ref{diraclemma}, to the (graph) path 
\[P' = v_1, e_1, \ldots, v_c, f_0, \ldots, f_\ell, u_{\ell}\] in $I_H$ with $z = u_{j_1}$. After modification as detailed in the proof of Lemma~\ref{maindirac}, we obtain two disjoint, aligned (Berge or partial Berge) paths $P_1'$ and $P_2'$ in $H$ such that $P_1'$ starts in $X$, ends in $u_{j_1}$ and is internally disjoint from $X \cup Y$, and $P_2'$ starts in $X$, ends in a vertex $u_i \in \{u_{j_1+1}, \ldots, u_{\ell}\}$ and is internally disjoint from $X \cup Y$ except possibly in its last edge, only if that edge is $f_{i-1}$. 

Suppose $a_1$ and $b_1$ are the first elements of $P_1'$ and $P_2'$ respectively, and let $Q$ be a long $a_1,b_1$-segment of $C$ provided by Claim~\ref{longQ}. If $a_1$ and $b_1$ are both vertices, then $|V(Q)| \geq c-(k-2)$, otherwise if at least one of $a_1$ or $b_1$ is in $E(C)$, then $|V(Q)| \geq c-(k-1)$. 

Let $\alpha$ be the largest index smaller than $i$ such that $u_\alpha \in e$. Consider the cycle
\[C^* = a_1, Q, b_1, P_2', u_i, P[u_i, u_{\ell}], u_{\ell}, e, u_{\alpha}, P[u_{\alpha}, u_{j_1}], u_{j_1}, P_1', a_1.\]

This cycle contains the set $U=\{u_{j_1}, \ldots, u_{\alpha}\} \cup \{u_i, \ldots, u_{\ell}\}$, and $e \subseteq U$. Since $j_1 > 1$,  $|U \cap V(C)| = 0$. 
 If $|U| > |e| = k-1$, then $|C^*| \geq c-(k-1) + |U| > c$, contradicting the choice of $C$ as a longest cycle. Therefore we may assume $e = U$.

We will show that \begin{equation}\label{1interval}
\hbox{\em $e$ contains $r$ consecutive vertices in $P$.}
\end{equation}

If not then $i > \alpha+1 \geq j_1 + 1$, and so $u_{i-1} \notin e$. Recall that $u_{\ell} \in f_{j_1}$, and $\{u_i, u_{i+1}\} \subset U = e$. 

If $f_{i-1} \notin E(P_2')$, set
\[C' = a_1, Q, b_1, P_2', u_i, f_{i-1}, \ldots, u_{j_1+1}, f_{j_1}, u_{\ell}, f_{\ell-1}, \ldots, u_{i+1}, e, u_{j_1}, P_1', a_1.\]

Otherwise set 
\[C' = a_1, Q, b_1, P_2'[b_1, f_{i-1}], f_{i-1}, u_{i-1}, f_{i-2}, \ldots, u_{j_1+1}, f_{j_1}, u_{\ell}, f_{\ell-1}, \ldots, u_i, e, u_{j_1}, P_1', a_1.\]

The cycle $C'$ contains $e \cup \{u_{i-1}\}$. Therefore $|C'| \geq c-(k-1) + (k-1 + 1) > c$, a contradiction. This proves~\eqref{1interval}. In particular, $u_{\ell-1} \in e$. 

By the choice of $(C,P)$ as a best lollipop, if we swap edge $e$ with $f_{\ell-1}$ in $P$ we obtain another lollipop with $f_{\ell-1}$ taking the place of $e \in E(H')$. By the case, $e \subseteq V(P) - V(C)$ and so by Rule (R4) in the choice of $(C,P)$, $f_\ell \subseteq V(P) - V(C)$ as well. In particular, the first vertex of $f_{\ell-1}$ in $P$ is not $u_0$. Thus by symmetry, $f_{\ell-1}$ also contains $r$ consecutive vertices in $P$. But this implies $f_{\ell-1} = e$ since both edges end in $u_{\ell}$, a contradiction.



{\bf Case 2}: In every best lollipop, $j_1 = 0$. By Claim~\ref{allneighbors}, all best lollipops are o-lollipops. 

Recall that for every eligible $u_{i+1}$, the lollipop $(C,P_{i+1})$ is a best lollipop and the path $P_{i+1}$ begins with $u_0, f_0$. By the case, $u_0 \in N_{H'}(u_{i+1})$, and so by the ``moreover'' part of Corollary~\ref{smalldeg} applied to $(C,P_{i+1})$, $u_{i+1} \in f_0$ for every eligible $u_{i+1}$. Thus \[r= |f_0| \geq |\{u_0\} \cup \{u_{i+1}: u_{\ell} \in f_i\}| \geq 1 + d_{P}(u_{\ell}) =k,\]
which contradicts that $r= k-1$. 
\end{proof}

\section{Proof of Lemma~\ref{nlem2}: short paths in lollipops imply short paths in dcp-pairs}

%
We have shown that the paths in best lollipops must be short. Recall that a disjoint cycle-path pair, or a dcp-pair, is a cycle and a path that share no defining vertices or edges.  We now show that all $2$-good dcp-pairs must also have short paths. This will be useful for us because it implies that since $r$ is large, edges intersecting the path must ``stick out" of the path.
We restate the contrapositive of Lemma~\ref{nlem2} in the following slightly stronger form.

{\bf Lemma~\ref{nlem2}$'$.} {\em Suppose $2\leq s\leq k\leq r+1$.
Let $(C,Q)$ be a $2$-good dcp-pair in a 2-connected $n$-vertex $r$-graph $H$
with $c(H)<\min \{n,2k,|E(H)|\}$ and $\delta(H)\geq k$.
If $|V(Q)|\geq s$, then $H$ contains a lollipop $(C,P)$  with  $|E(P)|\geq s$.}

\medskip
\begin{proof} Suppose $(C,Q)$ is  a $2$-good dcp-pair, say
$Q=u_1,f_1,u_2,\ldots,f_{q-1},u_{q}$ with $q\geq s$. Let $H'$ be the subhypergraph with edge set $E(H) - E(C) - E(Q)$.  If the lemma does not hold, then
\begin{equation}\label{noC}
\mbox{ \em no edge in $E(H)-E(Q)$ containing $u_1$ or $u_{q}$ intersects $V(C)$,}
\end{equation}
otherwise we already obtain a desired o-lollipop or p-lollipop.
By the maximality of $|V(Q)|$,
\begin{equation}\label{noCH}
\mbox{ \em 
  each $g\in E(H')$ containing $u_1$ or $u_{q}$
 does not meet $V(H)-V(C)-V(Q)$, thus by~\eqref{noC}, $g\subseteq V(Q)$
 .}
\end{equation}
 Let $R=w_1,g_1,w_2,\ldots,g_{t-1},w_t$ be a shortest path in $H$ from $V(C)$ to $V(Q)$. If $g_{t-1}=f_j \in E(Q)$ for some $j$ then we let $w_t=u_{j}$.
Rename the vertices of $C$ so that $w_1=v_c$ and $w_t=u_{j}$ for some $1\leq j\leq q-1$. Combining with $R$ any path with $s$ vertices starting with $u_j$ that is otherwise disjoint from $C$ and $R$ yields a lollipop as desired. 
If $j\geq s$, we take the path $Q[u_j, u_1]$. So, $j\leq s-1$.

{\bf Case 1:}  $d_{H'}(u_1)=0$. Let $j'$ denote the maximum $i$ such that $u_1\in f_{i}$. By the case, $j'\geq d_H(u_\ell) \geq k\geq s>j$, so  path
$u_j,f_{j-1},\ldots,f_1,u_1,f_{j'},u_{j'},f_{j'-1},\ldots,u_{j+1}$ has at least $k \geq s$ vertices and does not use $f_j$.

{\bf Case 2:} 
 $d_{H'}(u_1)\geq 2$ or the unique edge $h\in H'$ containing $u_1$ is not $\{u_1,\ldots,u_r\}$.
Let $j'$ denote the maximum $i$ such that some edge  $h\in H'$ contains $u_1$ and $u_i$.
  Then
  $j'\geq r+1>j$, so the path
$u_j,f_{j-1},\ldots,f_1,u_1,h,u_{j'},f_{j'-1},\ldots,u_{j+1}$ has at least $r+1\geq k$ vertices and does not use $f_j$.

{\bf Case 3:} 
The unique edge $h\in H'$ containing $u_1$ is  $\{u_1,\ldots,u_r\}$. Now, $f_1\neq \{u_1,\ldots,u_r\}$, and by~\eqref{noCH}, $f_1\subset V(Q)$.  So, switching $h$ with $f_1$, we get Case 2.
\end{proof}


\section{Proof of Lemma~\ref{nlem0}: Disjoint cycle-path pairs have non-trivial paths}
\label{secl1}

In this section, we show that the path  $P$ in a $3$-good dcp-pair
must have at least 2 vertices. 
Let $(C,P)$ be such a dcp-pair. If $|C| < n$, there is a vertex outside of $C$, thus $\ell \geq 1$.

We will show that $\ell \geq 2$ using the notion of {\em expanding} sets that can be used to modify $C$ into a longer cycle.
Indeed, suppose $\ell=1$ and $P= u_1$.

Let $C'$ be any longest cycle in $H$ (i.e., $|C'| = |C| = c$), and let $u \in V(H) - V(C')$.  
 Say that a set $W\subseteq V(C')$ is {\em $(u,C')$-expanding} 
if for every distinct $v_{i},v_{j}\in W$, there is a $v_{i},v_{j}$-path $R(v_{i},v_{j})$
whose internal vertices are disjoint from $V(C')\cup \{u\}$ and all edges are in $E(H)- E(C')$. One example of a  $(u,C')$-expanding set is $V(C')\cap g$ where $g$ is any edge in $E(H)- E(C')$. 
\begin{equation}\label{expset}
\mbox{\em Another  example is a set of the form $N_{H-C'}(w)\cap V(C')$ for a vertex $w\in V(H)-V(C')-\{u\}$.}
\end{equation}
\begin{claim}\label{no2}Let $(C',u)$ be a 2-good dcp-pair (so $C'$ is a longest cycle in $H$) with $C' = v_1, e_1, \ldots, v_c, e_c, v_1$.
For any  $(u,C')$-expanding set $W$,
$u$ is contained in at most one edge of $\{e_j: v_j \in W\}$ and in at most one edge of $\{e_{j-1}: v_j \in W\}$.

Moreover, suppose $v_j \in W$ and $u \in e_j$. For every $v_i \in W - \{v_j\}$, if $e \in E(H) - E(C') - E(R(v_i, v_j))$ contains $u$, then $v_{i+1} \notin e$. Similarly, if $u \in e_{j-1}$, then $v_{i-1} \notin e$.  
\end{claim}
\begin{proof}
Suppose $v_i, v_j \in W$ and $u \in e_i, e_j$. By symmetry we may assume $i < j$. The cycle
\[C''=v_1, \ldots, v_i, R(v_i, v_j), v_j, e_{j-1}, \ldots, v_{i+1}, e_i, u, e_j, v_{j+1}, e_{j+1}, \ldots, v_c, e_c, v_1\] has length $|C'| + 1 > c$, a contradiction. The case for $u \in e_{i-1}, e_{j-1}$ is symmetric.

For the ``moreover" part, let us suppose $e \in E(H) - E(C') - E(R(v_i, v_j))$ contains $u$ and $v_{i+1}$. Then replacing edge $e_i$ in $C''$ with $e$ also yields a  cycle longer than $C'$. 
\end{proof}
%
%
We now apply this claim to the 3-good dcp-pair $(C,P)$ with $C = v_1, e_1, \ldots, v_c, e_c, v_1, P = u_1$. Define 
 $A=N_{H'}(u_1)$, $a=|A|$, $B = \{e_j \in E(C), u_1 \in e_j\}$, and $b = |B|$. We obtain the following.

\begin{claim} \label{no3}Suppose $W$ is a $(u_1, C)$-expanding set.

(i) If the edges of $B$ form exactly $q$ intervals in $C$, then $|W| \leq c-(b+q) +2$.

(ii) If the vertices in $W$ form exactly $q'$ intervals in $C$, then $b \leq c-(|W|+q') + 2$.
\end{claim}

\begin{proof}
Set $V = \bigcup_{e_i \in B} \{v_i, v_{i+1}\}$. We have $|V| = b + q$. By Claim~\ref{no2}, $W$ contains at most one vertex $v_i$ for which $e_i \in B$ and at most one vertex $v_{i}$ for  which $e_{i-1} \in B$. Thus $|W \cap V|  \leq 2$ and $|W| \leq c - |V| + 2$. This proves (i). The proof of (ii) is similar.
\end{proof}


Now we are ready to prove Lemma~\ref{nlem0}.  Recall the statement.

{\bf Lemma~\ref{nlem0}.} {\em 
Let $(C,P)$ be a $3$-good dcp-pair in $H$. If $|V(C)|=c$, $|V(P))|=\ell$ and
 $c < \min\{2k,n,|E(H)|\}$, 
then $\ell \geq 2$.
}


\begin{proof} Suppose $P=u_1$. By Rule (R3),  $u_1$ is a vertex in $V(H)-V(C)$   with the largest degree in $E(C)$.
 If there exists $ e \in E(H) - E(C)$ containing at least 2 vertices $u, u' \notin V(C)$, then  
the path $P'=u,e,u'$, the pair $(C, P')$ is a better dcp-pair than $(C,P)$. It follows that
\begin{equation}\label{1c}
\hbox{for each $e \in E(H) - E(C)$, $|e \cap V(C)| \geq r-1$.}
\end{equation}

In particular, this yields
 $N_{H'}(u_1) \subseteq V(C)$.
By Claim~\ref{bigsmallcycle}, $A$ does not intersect the set $\bigcup_{e_i \in B} \{v_i, v_{i+1}\}$
and
 \begin{equation}\label{vcr} 
\mbox{\em no two vertices of $A$ are consecutive on $C$; thus $ 2k-1\geq c\geq 2a+b$.}
\end{equation}  
 It follows that
 \begin{equation}\label{abr} 
a \leq \lfloor c/2 \rfloor \text{ and } k\leq d_H(u_1)\leq {a\choose r-1}+b\leq {a\choose r-1}+c-2a.
\end{equation}

 \medskip
{\bf Case 1:}  $d_{H'}(u_1)\geq 2$. Then   $a\geq (r-1)+1$. If $r \geq k$, then $\lfloor c/2 \rfloor \geq a \geq r \geq k$, contradicting to $c < 2k$. Thus we may assume $k=r+1$. Now~\eqref{vcr} yields $a=r$, $c=2k-1$ and $b\leq 1$. On the other hand,~\eqref{abr} yields $b\geq 1$.
So, $b=1$ and all  $r$-tuples of vertices containing $u_1$ and contained in $A\cup \{u_1\}$ are edges of $H'$. By symmetry, we may assume that $B=\{ e_{2k-1}\}$ and $A=\{v_2,v_4,\ldots,v_{2k-2}\}$. Since $k= r+1$, for every $1\leq j\leq k-1$ we can choose an edge $g_{2j}\in E(H')$ containing $u_1$ and $v_{2j}$ so that all $g_{2j}$ are distinct. We also let $g_1=g_{2k-1}=e_{2k-1}$.

\begin{claim}\label{oddneighbors}Let $1 < i < 2k-1$ be odd. Then (i) the only edge in $E(C)-\{e_i, e_{i-1}\}$ that may contain $v_i$ is $e_{2k-1}$, and (ii) $N_{H'}(v_i) \cap V(C) \subseteq A$. 
\end{claim}

\begin{proof}
Suppose $v_i \in e_j$ for some $j$. By symmetry, we may assume $j > i$. If $j$ is even then set 
$$C_j = v_i, e_i, \ldots, v_j, g_j, u_1, g_{i-1}, v_{i-1}, e_{i-1}, \ldots, v_{j+1}, e_j, v_i,$$
 and if $j$ is odd, then set 
 $$C_j' = v_i, e_{i-1}, v_{i-1}, \ldots, v_{j+1}, g_{j+1}, u_1, g_{i+1}, v_{i+1}, e_{i+1}, \ldots, v_j, e_j, v_i.$$
 
 Then each of  $C_j$ and $C_j'$ contains $c+1$ vertices, contradicting the choice of $C$. This proves (i). For (ii), suppose some $h
\in E(H')$ contains $v_i$ and $v_j$ where $j$ is odd. Then replace edge $e_j$ in the cycle $C_j'$ above with $h$ to obtain a cycle longer than $C$.
\end{proof}

Since $|e_{2k-1} - \{u_1\}| = r-1 = k-2$, there exists $v_{i}\notin e_{2k-1}$ where $1 < i < 2k-1$ is odd. By Claim~\ref{oddneighbors}, $k \leq \delta(H) \leq d_H(v_i) \leq 2 + d_{H'}(v_i)$, and therefore $d_{H'}(v_i) \geq k-2 =(r+1)-2 \geq 2$. Then $(r-1) + 1 \leq |N_{H'}(v_i)| \leq |A| = k-1$, and so $N_{H'}(v_i) = A$. 

Set $C' = v_1, e_1, \ldots, v_{i-1}, g_{i-1}, u_1, g_{i+1}, v_{i+1}, e_{i+1}, \ldots, v_1$. We consider $(C', v_i)$ which is a 2-good dcp-pair that does not use the edges $e_{i-1}, e_i$. If one of these edges, say $e_i$, contains a vertex $w \notin V(C) \cup \{v_i\}$, then $(C', u_1, e_i, w)$ is a better p-lollipop than $(C,P)$. 

On the other hand, if $e_i \cup e_{i-1} \subseteq A$, then $d_H(v_i) \leq { |A| \choose r-1}  < k \leq \delta(H)$, a contradiction. By the previous claim, the only vertices in $V(C) - A - \{v_i\}$ that may belong to $e_i$ or $e_{i-1}$ are $v_1$ and $v_{2k-1}$. Without loss of generality, let $v_1$ belong to $e_j \in \{e_{i-1}, e_i\}$. Since $v_2 \in A$, there exists $g \in E(H')$ containing $\{v_2, v_i\}$ (and note that $g \neq g_{i-1}, g_{i+1}$ because $v_i \notin N_{H'}(u_1)$). Then replacing in $C'$ the segment $v_1,e_1,v_2$ with $v_1, e_j, v_i, g, v_2$ yields a longer cycle than $C$.

\medskip
{\bf Case 2:} $d_{H'}(u_1)=1$, say  $u_1\in f\in E(H')$. Then $d_H(u_1)=1+b$, $A=f\cap V(C)$ and $a=r-1$. 
  By~\eqref{abr},
$$k\leq d_H(u_1)\leq 1+c- 2(r-1)\leq 2k-2r+2.$$
 When $k > 2$, this is possible only if $k = r+1$, $r=3$, $c=2r+1=7$ and $d_H(u_1)=k=4$.
In this case we have $c = 2k-1 = 7$. 
Let $e \in E(H')$ contain $u_1$, say $e = \{u_1, v_{i}, v_j\}$. Since by Claim~\ref{bigsmallcycle}, $e$ cannot contain consecutive vertices in $C$,  by symmetry we may assume that $i=1$ and $j \in\{5,6\}$. By Claim~\ref{bigsmallcycle}, $u_1$ cannot be contained in the 4 distinct edges $e_1, e_7, e_j, e_{j-1}$ incident to $v_1$ and $v_j$. Thus as $d_H(u_1) \geq k = 4$ and $d_{H'}(u_1) = 1$, $u_1$ is contained in the remaining $7-4 = 3$ edges of $C$. That is, $u_1$ belongs to the edges $e_2, e_3$, and $e_k$ where $k=6$ if $j=5$ and  $k = 4$ if $j=6$. Let $C'$ be the cycle obtained by swapping $v_3$ with $u_1$ and observe that $(C', v_3)$ is a 2-good dcp-pair.  

The set $e=\{v_1, v_j, u_1\}$ forms a $(v_3,C')$-expanding set. Because $u_1 \in e$ now takes the place of $v_3$ in $C'$ and $v_3 \in e_2, e_3$, by Claim~\ref{no2} $v_3$ cannot belong to the edges $e_1, e_7, e_j, e_{j-1}$ and cannot be an $H'$-neighbor to $v_{2}, v_7, v_{j-1}$, or $v_{j+1}$ (here we use the fact that $v_3 \notin e$). Moreover by Claim~\ref{bigsmallcycle}, $N_{H'}(v_3)$ also does not contain vertices $u_1$ and $v_4$ which are incident to $e_2$ and $e_3$ respectively. 

This yields the following two facts: the only  edge of $C'$ other than $e_2$ and $e_3$ that may contain $v_3$ is $e_k$, and the only $H'$-edge containing $v_3$ may be $\{v_3, v_1, v_j\}$. Recall that $u_1 \in e_k$ and therefore $e_k = \{v_k, v_{k+1}, u_1\}$ which does not contain $v_3$. It follows that $d_H(v_3) \leq d_C(v_3) + d_{H'}(v_3) \leq 2 + 1 < k = \delta(H)$, a contradiction.


\medskip
{\bf Case 3:} $d_{H'}(u_1)=0$ for some $u_1\in V(H)-V(C)$. 
So, 
  $d_H(u_1)=b$. 
  If $u_1$ is the only vertex outside of $C$, then $r\geq\lfloor \frac{n-1}{2} \rfloor$, so by Theorem~\ref{mainold2}, $H$ has a Hamiltonian cycle, a contradiction. Hence, 
  \begin{equation}\label{c+2}
  \mbox{\em $n= c+x$\qquad for some $x\geq 2$.} 
  \end{equation}

  Let $w\in V(H)-(V(C)\cup \{u_1\})$. 
  Let us show that
  \begin{equation}\label{d(w)}
d_{H'}(w)\leq 1. 
\end{equation}
  Indeed, suppose $g_1,g_2\in E(H')$ and $w\in g_1\cap g_2$. Let $W=V(C)\cap (g_1\cup g_2)$. As observed in~\eqref{expset}, this $W$ is $u_1$-expanding. Since $g_2\neq g_1$, by~\eqref{1c}, $|W|\geq r$. Also, by Claim~\ref{shortpath},
  vertices in $g_2$ could not be next to  vertices in $g_1$ on $C$. Thus if $|W|= r$, then $|g_1\cap g_2|=r-1$; hence no two vertices  of $W$ are consecutive on $C$. In this case, by Claim~\ref{no3}(ii),
 $b\leq c-|W|-q+2$ where $q=|W|=r$.  Since $k\leq r+1$,  we get
 $k\leq  d_H(u_1)\leq (2k-1)-2r+2\leq 2k-r-2$ contradicting $k\leq r+1$. 
 
 Thus $|W|\geq r+1$. But still since  vertices in $g_2$ could not be next to  vertices in $g_1$ on $C$, $q\geq 2$.
So,
 $k\leq  d_H(u_1)\leq (2k-1)-(r+1+2)+2= 2k-r-2$ contradicting $k\leq r+1$. This proves~\eqref{d(w)}.

 Let $n=c+x$ and $|E(H)|=c+y$. Then considering the sum of degrees of vertices in $H$, we have
  \begin{equation}\label{c+x}
  k(c+x)=k\cdot n\leq \sum\nolimits_{v\in V(H)}d_H(v)=r(c+y).
  \end{equation}
 
 \medskip
By~\eqref{c+2}, if $k\geq r$, then 
  \begin{equation}\label{+2}
y\geq x\geq 2.
\end{equation}
 Moreover, if $k=r+1$ then~\eqref{c+x} yields
$$y\geq \frac{(r+1)(c+x)}{r}-c= \frac{c+x}{r}+x=\frac{n}{r}+x>2+x.
$$
 So, if $Z=\{g_1,\ldots,g_z\}$ is the set of edges in $E(H')$ contained in $V(C)$,
  then~\eqref{d(w)} together with  $d_{H'}(u_1)=0$ yields 
   \begin{equation}\label{four}
\mbox{\em $y\leq z+x-1$, and hence $z\geq 1$ when $k\geq r$ and $z\geq 4$ when $k= r+1$.}
\end{equation}

 Suppose the edges of $B$ form exactly $q$ intervals in $C$.  Since $b=d_H(u_1)\geq k$, by Claim~\ref{no3}(i)
 $$r\leq |W| \leq c - (b+q) + 2 \leq 2k-1 -(k+q) + 2 =k-q+1.$$
  As $k \leq r+1$, $q \leq 2$ with equality only if $k=r+1$ and $d_H(u_1) = k$. Let $J_B = \{j\in[c]: e_j \in B\}$. By symmetry we may assume $1 \in J_B$. 
%
%
  
  {\bf Case 3.1:} $J_B=\{1,\ldots,i_1\}\cup \{i_2,\ldots,i_3\}$ where 
  $ i_1+2\leq i_2\leq i_3\leq c-1$.
   In this case $q = 2$, so $k= r+1$ and $|J_B| = d_H(u_1) = k$. 
If some $g_j\in Z$ contains $v_i$ for some $i$ in the $k-2$-element subset $\{2,\ldots,i_1\}\cup \{i_2+1,\ldots,i_3\}$, then in order not to create a longer cycle, the indices of the  remaining $r-1$ vertices of $g_j$ cannot be in $J'_B= \{1,2,\ldots,i_1+1\}\cup \{i_2,\ldots,i_3+1\}$ by Claim~\ref{no2}.
Since $|J'_B|=k+2$, we need $c\geq (r-1)+(k+2)=2k$, a contradiction. So, each $g_j\in Z$ is contained in the set
$F=\{v_{i_1+1},\ldots,v_{i_2}\} \cup \{v_{i_3+1},\ldots,v_c,v_1\}$ which has size $c - (k-2) = k+3 = r+2$. Also, each $g_j$ contains at most one of $v_1,v_{i_2}$ and 
at most one of $v_{i_1+1},v_{i_3+1}$. There are exactly four such possibilities: $F-\{v_1, v_{i_1+1}\},
F-\{v_1, v_{i_3+1}\},F-\{v_{i_2}, v_{i_1+1}\},F-\{v_{i_2}, v_{i_3+1}\}$. So, $z=4$ and each $g_j$ is one of these possibilities, say $g_1=F-\{v_{i_2}, v_{i_1+1}\}$.
But then $g_1$  contains $\{v_c,v_1\}$, and we can replace $e_c$ with $g_1$, and again get a $3$-good dcp-pair. Since $g_1\neq e_c$, this is a contradiction.

  {\bf Case 3.2:} $J_B$ is a single interval for each $3$-good lollipop. Let
  $J_B=\{1,\ldots,b\}$ for $b \geq k$. 
%

First suppose $k \leq r$. Then $z \geq 1$, say $g_1 \in Z$.    
     If  $g_1$ contains a vertex $v_j$ with $2 \leq j \leq b$, then $g_1-v_j \subseteq \{v_{b+2}, \ldots, v_c\}$ by Claim~\ref{no2}. It follows that
   $r-1=|g_1-v_j|\leq c-(b+1)\leq (2k-1)-(k+1)=k-2$, contradicting $k\leq r$. Hence $g_1\subseteq \{v_{b+1},v_{b+2},\ldots,v_c,v_1\}$, and so $r=|g_1|\leq c-(b-1)\leq (2k-1)-(k-1)=k$. This is possible only it $c=2k-1$, $b=k$, $k=r$ and
   $g_1=\{v_{b+1},v_{b+2},\ldots,v_c,v_1\}$. Let $C'$ be obtained from $C$ by replacing $e_{c-1}$ with $g_1$. Then $(C',P)$ is also a $3$-good dcp-pair. We apply~\eqref{four} to $(C',P)$ to obtain that there exists some $g_1' \in E(H) - E(C')$ contained in $V(C)$
 (possibly $g_1'=e_{c-1}$). Then we contradict Claim~\ref{no2}, since $g_1'\neq g_1$.
   

Finally consider the case $k= r+1$. By~\eqref{four}, $z\geq 4$.
  First suppose $b \geq k+1$. If some $g_i \in Z$ contains a vertex $v_j$ with $2 \leq j \leq b$, then $g_i-v_j \subseteq \{v_{b+2}, \ldots, v_c\}$ by Claim~\ref{no2}. That is, $r= |g_i| \leq 1 + (c - (b+2) + 1) \leq k-2$, a contradiction. Thus each $g_i \in Z$ is a subset of $\{v_{b+1}, \ldots, v_c, v_1\}$ which has size at most $k-1= r$. We get a contradiction to $z \geq 4$.
  
 So we may assume $J_B = \{1, \ldots, k\}$.  Let $F=\{v_{k+2},\ldots,v_{c}\}$, $B'=E(C)-E(B)$
  and $B''=B'-e_{k+1}-e_c=\{e_{k+2},\ldots,e_{c-1}\}$. 
  
  Among all $c$-cycles $C$  with the same cyclic sequence $v_1,\ldots,v_c$ of vertices and containing edges $e_1,\ldots,e_k$ in the same positions, choose one in which
  
 (a) most edges in $B'$ are subsets of  $F'=F\cup \{v_{k+1},v_1\}=\{v_{k+1},\ldots,v_{c},v_1\}$ and
 
 (b) modulo (a), most edges in $B''$ contain $F$.
  
We claim that 
\begin{equation}\label{H'-Z}  
  \mbox{\em no edges in $E(H')-Z$ intersect $\{v_2,\ldots,v_k\}$.}
  \end{equation}
Indeed,  suppose  some $g_0\in E(H')-Z$ contains $v_i$ for some $2\leq i\leq k$. By~\eqref{1c}, there is a unique
$w\in g_0-V(C)$.    By~\eqref{d(w)}, the set $B_w$ of the edges of $C$ containing $w$ has at least $k-1$ edges.
By Claim~\ref{no2}, there is no  $i'\in \{1,\ldots,k+1\}-\{i\}$ such that $v_{i'}\in g_0$. Hence $g_0-v_i-w\subseteq F$. Since $|F|=r-1=|g_0-v_i-w|+1$, we may assume by symmetry that $v_c\in g_0$. Recall that if $v_j\in g_0$ then $e_{j-1},e_j \notin B_w$. This yields that $B_w\subseteq 
\{e_1,\ldots,e_{k+1}\}$. So, there is $1\leq i'\leq k$ such that $e_{i'}\in B_w$. As above, $i'\notin \{i-1,i\}$. By symmetry we may assume $i'>i$. Then  the cycle \[ v_1, e_1, v_2, \ldots, v_{i-1},e_{i-1}, u_1, e_{i'-1}, v_{i'-1}, e_{i'-2}, v_{i'-2}, \ldots,  v_{i}, g_0, w, e_{i'}, v_{i'+1}, e_{i'+1}, \ldots,  v_c, e_c, v_1
\]
omits $v_{i'}$ but goes through $w$ and $u_1$, thus is longer than $C$. This contradiction proves~\eqref{H'-Z}.

Next, we claim that 
\begin{equation}\label{H'+Z}  
  \mbox{\em no edges in $Z$ intersect $\{v_2,\ldots,v_k\}$.}
  \end{equation}
Indeed, suppose  some $g_1\in Z$ contains $v_i$ for some $2\leq i\leq k$. Then
   the indices of the  remaining $r-1$ vertices of $g_1$ cannot be in $J'_B= \{1,2,\ldots,k+1\}$.
   So, $g_1-v_i\subseteq F$, which yields $c=2k-1$, $k=r+1$, and $g_1=F+v_i$. Now, by choice of $C$, since we can switch $g_1$ with any  edge in $B''$,  each edge in $B''$ either is contained in $F'$ or  contains $F$. Since only one edge containing $F$ may also contain $v_i$,
  \begin{equation}\label{H'+Z'}  
  \mbox{\em  
   $v_i$ does not belong to any edge in $B''\cup (Z-g_1)$.  }
  \end{equation}
  Let cycle $C'$ be formed by swapping $u_1$ with $v_i$ and $g_1$ with  $e_{c-1}$, that is, 
\[C' = u_1, e_i, v_{i+1}, e_{i+1}, v_{i+2}, \dots , e_{c-2}, v_{c-1}, g_1, v_{c}, e_{c}, v_{1},\ldots,  e_{i-1}, u_1.\] 
By~\eqref{H'-Z} and~\eqref{H'+Z'},    $v_i$ is contained only in edges of $C'$, so $C'$ is an optimal choice of cycle under the same conditions as $C$. If the edges of $C'$ containing $v_i$ are not all consecutive  along $C'$, then we get Case 3.1 and arrive at a contradiction.
So they are consecutive. As $v_i$ is contained in $g_1$ (which plays the role of $e_{c-1}$ in $C'$), either $v_i \in e_c$ or $v_i \in e_{c-2}$. If $v_i \in e_c$ then the cycle
 \[C''= v_1, e_1, v_2, \ldots, v_{i-1},e_{i-1}, u_1, e_{i}, v_{i+1}, e_{i+1}, v_{i+2}, \ldots, v_{c-1}, e_{c-1}, v_{c}, g_1, v_i, e_{c},  v_1
\]
is longer than $C$. The case for $v_i \in e_{c-2}$ is similar. This proves~\eqref{H'+Z}.

By~\eqref{H'+Z}, all edges in $Z$ are contained in $F'$. Since $|Z|\geq 4$, each edge in $B''$ can be replaced with an edge in $Z$. Thus by Rule (a), each edge in $B''$ also is contained in $F'$. Then the $(r+1)$-element set $F'$ contains at least $|B''|+|Z|\geq r-2+4=r+2$ subsets of size $r$, a contradiction.
  \end{proof}

\section{Disjoint cycle-path pairs and cycle-cycle pairs}



We conclude the proof of Theorem~\ref{mainthm3} in this section by proving Lemma~\ref{nlem1}. Recall the statement.

{\bf Lemma~\ref{nlem1}.} {\em
Let $(C,P)$ be a $4$-good dcp-pair in $H$. If $|V(C)|=c$, $|V(P))|=\ell$ and
 $c < \min\{2k,n,|E(H)|\}$, 
then $\ell \geq k$.}

\medskip

When handling  dcp-pairs, we will also consider structures which are a bit richer,  called {\em disjoint cycle-cycle pairs}
or  {\em dcc-pairs} for short. A {\bf disjoint cycle-cycle pair}  $(C,P)$ is
obtained from a dcp-pair  $(C,P')$ with $P'=u_1, f_1, u_2, \dots , u_{\ell}$ by adding edge $f_\ell \notin E(C) \cup E(P')$ containing $\{u_1, u_\ell\}$. In other words, in a dcc-pair $(C,P)$, $P$ is a cycle $P=u_1, f_1, u_2, \dots , u_{\ell},f_\ell,u_1$ whose defining elements are disjoint from the defining elements of $C$.

We partially order dcp- and dcc-pairs together:
We say a pair $(C,P)$ {\em is better} than a pair $(C',P')$ if

  \begin{enumerate}
\item[(S1)]  $|V(C)|>|V(C')|$, or 
\item[(S2)] Rule (S1) does not distinguish $(C,P)$  from   $(C',P')$, and $|V(P)|>|V(P')|$; or

\item[(S3)] Rules (S1) and (S2) do not distinguish $(C,P)$  from   $(C',P')$, and the total number of vertices of $V(P)$ contained in the  edges of $E(C)$ counted with multiplicities is larger than the total number of vertices of $V(P')$ contained in the  edges of $E(C')$; or 

\item[(S4)] Rules (S1)--(S3) do not distinguish $(C,P)$ from $(C',P')$, $P$ is a cycle and $P'$ is a path; or

\item[(S5)] 
 Rules (S1)--(S4) do not distinguish $(C,P)$  from   $(C',P')$, and
 the number of edges in $E(P)$ fully contained in $V(P)$ is larger than the number of edges in $E(P')$ fully contained in $V(P')$. 
 \end{enumerate}

Since a dcc-pair is a dcp-pair $(C,P)$ with an additional edge added to $P$ and rules (S1)--(S3) are the same as  (R1)--(R3), all claims in Section~\ref{simple} that hold for $i$-good lollipops and dcp-pairs when $i\leq 3$ also hold for $i$-good 
dcc-pairs.

Let $(C,P)$ be a best dcp- or dcc-pair. Let $C=v_1, e_1, v_2, \dots, v_c$ and let $P= u_1, f_1, u_2, \dots u_{\ell}$ when $P$ is a path and $P= u_1, f_1, u_2, \dots u_{\ell},f_\ell,u_1$ when $P$ is a cycle. To prove Lemma~\ref{nlem1}, we consider the case $2 \leq \ell \leq k-1$, $3 \leq k \leq r+1$, and $\delta(H) \geq k$. 

Since $\ell \leq k-1\leq r$, at most one edge can be contained entirely in $V(P)$,
and this is possible only if $k=r+1$. Moreover, by (S5) in this case such an edge is in $E(P)$.
Thus any $H'$-edge containing $u_1$ or $u_\ell$ must be contained in $V(C) \cup V(P)$ and must intersect $V(C)$.
One benefit of dcc-pairs is that when $P$ is a cycle, then any two consecutive vertices can play the role of $u_1$ and $u_\ell$. In particular,
\begin{equation}\label{cyc1} \mbox{\em if $P$ is a cycle, then
 each  $h\in E({H'})\cup E(P)$ intersecting $V(P)$ is contained in $V(P)\cup V(C)$.}
\end{equation}

For all $j\in [\ell]$, we denote by $B_j$ the set of edges of $C$ containing $u_j$, and let $b_j=|B_j|$.

In the next subsection we prove a series of claims about vertices and edges in graph paths and cycles. We will apply these claims to the projections of the Berge cycle $C$ (that is, the graph cycle $v_1, e_1, \ldots, v_c, e_c, v_1$ where $e_i = v_iv_{i+1}$). We complete the proof of Lemma~\ref{nlem1} in the final two subsections, breaking into cases of dcp-pairs and dcc-pairs.

\subsection{Useful facts on graph cycles}\label{graph-c}


\begin{claim}\label{ver-new} Let $C=v_1,e_1,\ldots,v_s, e_s, v_1$ be a graph cycle. Let $A$ and $B$ be nonempty subsets in $V(C)$ such that 
\begin{equation}\label{dis}
\mbox{ for each $v_i\in A$ and $v_j\in B$, either $i=j$ or $|i - j| \geq q\geq 2$. }
\end{equation}

(i)  If $A=B$, then $s\geq q|A|.$ 

(ii) If $B\neq A$, then $s \geq |A| + |B| + 2q-3$ with equality only if $A\subset B$ or $B\subset A$.
\end{claim}

\begin{proof} Part (i) is obvious. We prove (ii) by induction on $|A\cap B|$ and take indices modulo $s$.

If $A\cap B=\emptyset$, then $C$ contains $|A|+|B|$ vertices in $A\cup B$ and at least $q-1$ vertices outside of 
$A\cup B$ between $A$ and  a closest to $A$ vertex in $B$ in either direction along the cycle.

Suppose now that (ii) holds for all $A'$ and $B'$ with $|A'\cap B'|<t$ and that $|A\cap B|=t>0$, say $v_i\in A\cap B$.
By symmetry, we may assume $|A|\leq |B|$.  If $A=\{v_i\}$, then $C$ has $|B|-1$ vertices in $B-A$ and at least $q-1$
vertices between $v_i$ and a closest to $v_i$ vertex in $B-A$, in either direction along the cycle. Thus,
$s\geq 1+|B-A|+2(q-1)=|A|+|B|-1+2q-2 = |A| + |B| + 2q-3$, as claimed.

Finally, suppose $|A|\geq 2$. 
By definition, $(A\cup B)\cap \{v_{i-q+1}, v_{i-q+2},\ldots,v_{i+q-1}\}=\{v_i\}$.
Define $e=v_{i-q}v_{i+1}$ and let $A'=A-v_i$, $B'=B-v_i$, and 
$C'=v_1,e_1,\ldots,v_{i-q},e,v_{i+1},e_{i+1},\ldots,v_s, e_s, v_1$. Then $A'$ and $B'$ satisfy~\eqref{dis}.
So, by induction, $|V(C')|\geq |A'|+|B'|+2q-3$ with equality only if $A'\subset B'$. 
Hence
$$s\geq q+|V(C')|\geq q+ (|A|-1)+(|B|-1)+2q-3=|A|+|B|+ 3q-5$$
 with equality only if $A\subset B$. 
 Since $q\geq 2$, this proves (ii).
\end{proof}

The line graph of cycle $C_s$ of length $s$ is again $C_s$ in which the vertices play the roles of the edges of the original graph. Thus Claim~\ref{ver-new} implies the following.

\begin{claim}\label{ed-new} Let $C=v_1,e_1,\ldots,v_s, e_s, v_1$ be a graph cycle. Let $A$ and $B$ be nonempty subsets in $E(C)$ such that 
\begin{equation}\label{edg}
\mbox{ for each $e_i\in A$ and $e_j\in B$, either $i=j$ or $|i - j| \geq q\geq 2$. }
\end{equation}

(i)  If $A=B$, then $s\geq q|A|.$ 

(ii) If $B\neq A$, then $s \geq |A| + |B| + 2q-3$ with equality only if $A\subset B$ or $B\subset A$.
\end{claim}

\begin{claim}\label{consecpath2}Let $C=v_1,e_1, \ldots, v_s, e_s, v_1$ be a graph cycle. Suppose $F\subset V(C)$. 
 Let $A$ and $B$ be nonempty subsets of $V(C)$ that are disjoint from $F$ and such that
\begin{equation}\label{dis2}
\mbox{ for each $v_i\in A$ and $v_j\in B$, either $i=j$ or $|i - j| \geq q\geq 2$. }
\end{equation}
\begin{equation}\label{dis3}
\mbox{ for each $v_i\in F$ and $v_j\in F \cup A \cup B$, either $i=j$ or $|i - j| \geq q\geq 2$. }
\end{equation}

(i)  If $A=B$, then $s\geq q|A|+q|F|$. 

(ii) If $B\neq A$, then $s \geq |A| + |B| +2q-3 +q|F|$ with equality only if $A\subset B$ or $B\subset A$.
\end{claim}

\begin{proof} Let $C'=v'_1, e'_1,v'_2, \ldots,v'_{s'}, e'_{s'}, v'_1$ be a cycle obtained from $C$
by iteratively contracting $e_i, e_{i+1}, \dots e_{i+q-1}$ for each $v_i \in F$. In particular, $s'=s-q|F|$.
Due to ~\eqref{dis3}, each $v_i \in A \cup B$ was unaffected by the edge contractions and hence still exists as some $v'_{i'}$ in $C'$. Moreover, by ~\eqref{dis3}, ~\eqref{dis2} still holds for $A$ and $B$ in $C'$.

So,  Claim~\ref{ver-new} applied to $A,B$ and $C'$ yields that if $A=B$, then $s'\geq q|A|$, and 
if $B\neq A$, then $s' \geq |A| + |B| + 2q-3$ with equality only if $A\subset B$ or $B\subset A$. 
Since $s'=s-q|F|$, this proves our claim.
\end{proof}

\begin{claim}\label{ver-ed} Let $q\geq 1$. Let $C=v_1,e_1,\ldots,v_s, e_s, v_1$ be a graph cycle. Let $I$ be an independent  nonempty subset of $V(C)$ and  $\emptyset\neq B\subset E(C)$ be
 such that 
\begin{equation}\label{v-e}
\mbox{ for each $v_i\in I$ and $e_j\in B$, the distance  on $C$ from $v_i$ to $\{v_j,v_{j+1}\}$ is at least $q$. }
\end{equation}
Then  $s\geq 2|I|+|B|+2(q-1)$. 
\end{claim}

\begin{proof} The vertices of $I$ partition the edges of $C$ into $|I|$ intervals of length at least $2$. If such an interval $Q$ contains an $e_j\in B$, then by~\eqref{v-e}, apart from edges of $B$ it also contains at least $2q$ edges from the ends of the interval to the closest edges in $B$. This proves the claim.
\end{proof}

\begin{claim}\label{ver-ed2} Let $q\geq 1$, $q' \geq q-1$. Let $C=v_1,e_1,\ldots,v_s, e_s, v_1$ be a graph cycle. Let $A$ be a nonempty subset of $V(C)$ and  $\emptyset\neq B\subset E(C)$ be
 such that 
\begin{equation}\label{v-e2}
\mbox{ for each $v_i\in A$ and distinct $e_j, e_k\in B$, the distance  on $C$ from $v_i$ to $\{v_j,v_{j+1}\}$ is at least $q'$, and $|j-k| \geq q$. }
\end{equation}
Then  $s\geq |A|+q|B|+2q'-q$. 
\end{claim}

\begin{proof} Removing the edges of $B$ from $C$ yields $B$ path segments. Any segment not containing a vertex in $A$ has length at least $q-1$. Any segment containing a vertex from $A$ contains at least $2(q'-1) \geq q-1$ edges not incident to a vertex in $A$, and there are at least $|A| + 1$ edges incident to $A$. Thus $s \geq |B|(q-1)+|B|+2(q'-1)-(q-1)+|A|+1$. 
\end{proof}
%
%
%

\begin{claim}\label{ver-ver} Let $q\geq 3$. Let $C=v_1,e_1,\ldots,v_s, e_s, v_1$ be a graph cycle. Let $I$ be an independent  nonempty subset of $V(C)$ and $\emptyset\neq A\subset V(C)$ be
 such that  for each $v_i\in I$ and $v_h\in A$,
\begin{equation}\label{v-e'}
\mbox{the distance  on $C$ from $v_i$  to   $v_h$ is either $0$ or at least $q$. }
\end{equation}
  (i) If $A=I$ then $s\geq q|I|$. 
  
   (ii) If $A\subsetneq I$ then $s\geq 2|I|+(q-2)(|A|+1)$. 
   
  (iii) If $A\setminus I\neq \emptyset$, then $s\geq 2|I| +|A|+2q-3$.   
  
\end{claim}

\begin{proof} Call a subpath of $C$ connecting two verties in $I$ and not containing other vertices in $I$ an $I$-{\em interval}. If $A=I$, then each $I$-interval has length at least $q$, which yields (i). 
To prove (ii), observe that each $I$-interval has at least  $2$ edges, and the $I$-intervals with at least one end in $A$ have length at least $q$. Since $A\neq I$, the number of such intervals is at least $|A|+1$. This proves (ii).

We prove (iii) by induction on $|A\cap I|$. If $|A\cap I|=\emptyset$, then
each  $I$-interval $Q$ containing $m>0$ vertices of $A$
  by~\eqref{v-e'} contains at least $2q$ edges from the ends of the interval to the closest vertices in $A$ and a path of at least $m-1$ edges connecting these extremal vertices. Together, $Q$ has at least $2q+m-1$ edges which is by $m+2q-3$ more than $2$. This proves the base of induction. 
  
  Suppose now $v_i\in A\cap I$. Consider $C'$ obtained from $C$ by deleting vertices $v_i,v_{i+1},\ldots,v_{i+q-1}$ and adding edge $v_{i-1}v_{i+q}$. Let $I'=I-v_i$ and $A'=A-v_i$. By induction assumption, 
  $|V(C')|\geq 2|I'|+|A'|+2q-3=2|I|+|A|+2q-6$. So, $s\geq 2|I|+|A|+2q-6+q\geq 2|I|+|A|+2q-3$, as claimed.
\end{proof}

\subsection{Proof of Lemma~\ref{nlem1} for dcp-pairs}
\begin{proof}[Proof of Lemma~\ref{nlem1} when $(C,P)$ is a dcp-pair.]Suppose that among all dcp- and dcc-pairs, a best $(C,P)$ is a dcp-pair.

 \textbf{Case 1:} $u_1$ is in an $H'$-edge $g$. If $u_\ell\in g$, then we can add $g$ to $P$ as $f_\ell$, a contradiction
 to (S4).
So, $|g\cap V(C)|\geq r-\ell+1$. 
 If $b_\ell>0$, then
by Claims~\ref{shortpath} and~\ref{ver-ed} with $I=g\cap V(C)$, $B=B_\ell$ and $q=\ell$, we have 
\[c\geq 2(r-\ell+1)+2(\ell-1)+b_\ell=2r+b_\ell\geq 2k-2+b_\ell.\] So $b_\ell\leq 1$, and $d_{H'}(u_\ell)\geq d_H(u_\ell) - b_\ell - d_P(u_\ell) \geq  k-1-(\ell-1)\geq 1$.
Let $g'\in E(H')$ contain $u_\ell$. Since $P$ does not extend to a cycle, $u_1\notin g'$, so $g'\neq g$ and  $|g'\cap V(C)|\geq  r-\ell+1$.

 For $j\in\{1,\ell\}$, let $A_j=N_{H'}(u_j)\cap V(C)$.
 By Claim~\ref{shortpath}, if $v_i, v_j \in A_1 \cap A_\ell$ then $i=j$ or $|i-j| \geq \ell+1$. We apply Claim~\ref{consecpath2} with $F=A_1\cap A_\ell,$
$A=A_1-F$,
$B=A_\ell-F$ and $q=\ell+1$. 
If  $A\not\subseteq B$ and $B\not\subseteq A$, then Claim~\ref{consecpath2}(ii) gives
$$c\geq 2(r-\ell+1-|F|) +2(\ell+1)-2 +(\ell+1)|F|\geq 2r-2\ell+2-2|F|+2\ell+3|F|\geq 2r+2\geq 2k,$$
a contradiction.  If say $A\subsetneq B$ but $|A\cup F|+|B\cup F|\geq 2r-2\ell+3$, then 
$$c\geq 2r-2\ell+3-2|F| +2(\ell+1)-3 +(\ell+1)|F|\geq 2r-2\ell+2+2\ell\geq  2k,$$ again. So, assume $A_1=A_\ell$. 

If $k\leq r$, then by Claim~\ref{consecpath2}(i),
  $c\geq (\ell+1)(r-\ell+1)\geq 2r\geq 2k$ for $2\leq \ell\leq k-1\leq r-1$, 
  a contradiction.
  Thus, let $k=r+1$. Also, if $|A_1|\geq r-\ell+2$, then we similarly have 
 $c\geq (\ell+1)(r-\ell+2)\geq 2(r+1)\geq 2k$. Thus  $|A_1|= r-\ell+1$, which means $g\cap V(P)=V(P)-u_\ell$,
 $g'\cap V(P)=V(P)-u_1$,
  and $u_1,u_\ell$ are not contained in other edges of $H'$. It follows that there are $e_i\in E(C)$ containing $u_1$
  and  $e_j\in E(C)$ containing $u_\ell$. 

If $i\neq j$, then Claim~\ref{ver-ed} with $I = A_1$, $B = \{e_i, e_j\}$ and $q =\ell$ implies $c \geq 2(r-\ell+1) + 2 + 2(\ell-1) =2r+2 \geq 2k$. Thus $e_i=e_j$ and $d_C(u_1) = d_C(u_\ell) = 1$ and $r = |e_i| \geq |\{u_1, u_\ell, v_i, v_{i+1}\}| \geq 4$. At least one of $f_1,f_{\ell-1}$ is not $V(P)$, say $f_1\neq V(P)$. Then switching $g$ with $f_1$ in $P$ we get another dcp-pair that is $2$-good. Hence every vertex in $f_1 \cap V(C)$ is distance at least $\ell+1$ from any vertex in $A_\ell =A_1$. Either $f_1\cap V(P)\neq \{u_1,\ldots,u_{\ell-1}\}$ or $A_1\neq f_1\cap V(C)$. In both cases, we get a case we have considered.

 
\medskip

\textbf{Case 2:}  $u_1, u_\ell$ are not in any $H'$-edges. Then $d_{C}(u_1)\geq k-(\ell-1)\geq 2$
and  $d_{C}(u_\ell)\geq k-(\ell-1)\geq 2$.

We apply Claim~\ref{ed-new} with 
$A=B_1$,
$B=B_\ell$ and $q=\ell$. If $A\not\subseteq B$ and $B\not\subseteq A$, then the claim gives
$c\geq 2(k-\ell+1)+2\ell-2=2k$, a contradiction. If say $A\subseteq B$ but $|A|+|B|\geq 2k-2\ell+3$, then 
$c\geq 2k-2\ell+3+2\ell-3=2k$, again. So assume $B_1=B_\ell$ and $|B_1|=k-\ell+1$. 

Then $d_P(u_\ell)\geq \ell-1=|E(P)|$. So $u_\ell\in f_1$. For any $2\leq i \leq \ell-1$ Consider $P_{i}=u_1,f_1, \ldots, u_{i-1}, f_{i}, u_\ell, f_{\ell-1}, \ldots, u_{i}$, and observe that $(C,P_{i})$ is another best dcp-pair. 
If $d_{H'}(u_{i})\geq 1$, then we get Case 1. Otherwise
as above, we get $B_{i}=B_1$ and $d_P(u_{i})=|E(P)|$. 
In other words,
\begin{equation}\label{allc}
V(P)\subseteq \bigcap\nolimits_{j=1}^\ell f_j\;\mbox{\em and $V(P)\subseteq e_i$ for each $e_i\in B_1=\ldots=B_\ell$.}
\end{equation}
By Claim~\ref{shortpath}(i), for distinct $e_i, e_j \in B_1$, $|j-i| \geq \ell$. 

Since $V(P)\subseteq f_j$ for each $j\in [\ell]$, we can construct any path covering $V(P)$ using edges
$f_1,\ldots,f_{\ell-1}$ in arbitrary order.  Let $I=\left(\bigcup_{j=1}^\ell f_j\right)\cap V(C)$ and
$J=\left(\bigcup_{j=1}^\ell f_j\right)-(V(P)\cup V(C))$. 

Suppose $w\in J$, by symmetry say $w\in f_1$. Assume first that there is $g\in E(H')$ such that $w\in g$. By the case,
$g\cap V(P)=\emptyset$.

Let $P'$ be obtained by replacing $u_1$ with $w$ in $P$. Then $(C,P')$ is a $2$-good dcp-pair and so by Claims~\ref{shortpath} and~\ref{allneighbors}, $g \subseteq V(C) \cup \{w\}$ and the distance from any $v_i \in g$ to some edge in $B_1$ is at least $\ell$. Apply Claim~\ref{ver-ed2} with $A = g \cap V(C)$, $B = B_1$, and $q=q'=\ell$ to obtain $c \geq (r-1) + \ell (k-(\ell-1)) + \ell \geq (r-1) + 2k  > c$, a contradiction. Thus $d_{H'}(w) = 0$. 
%

If there is $e_{j}\in E(C)-B_1$ such that $w\in e_{i_0}$, then in view of the path $P'$ above, then for any $e_i \in B_\ell,  |j-i| \geq \ell$ by Claim~\ref{shortpath}. Therefore
$c\geq \ell(|B_1| +1) = \ell (k-(\ell-1)+1)\geq 2k$, a contradiction again.
 Since $d_H(w)\geq k$, $w\in \bigcap\nolimits_{j=1}^{\ell-1} f_j$ 
and $w\in e_i$ for each $e_i\in B_1$. Then ~\eqref{allc} holds not only for $V(P)$ but for $V(P)\cup J$, which implies
$\ell\leq |V(P)\cup J|\leq r-2$ and therefore
 $|I|\geq 2$.
If $v_i\in I$, say $v_i\in f_1$, $i\leq i'\leq i+\ell-2$ and $e_{i'}\in B_\ell$, then the cycle obtained from $C$ by replacing subpath $v_i,e_i,v_{i+1},\ldots,e_{i'},v_{i'+1}$ with the path $v_i,f_1,u_2,f_2,\ldots,u_\ell,e_{i'},v_{i'+1}$ is longer than $C$,
a contradiction.

Applying Claim~\ref{ver-ed2} with $A = I, B = B_1, q' = \ell-1, q=\ell$ we obtain $c \geq 2 + \ell(k-(\ell-1)) + 2(\ell-1)-\ell \geq \ell(k-\ell+2) \geq 2k$, a contradiction.
\end{proof}

\subsection{Proof of Lemma~\ref{nlem1} for dcc-pairs}


\begin{proof}[Proof of Lemma~\ref{nlem1} when $(C,P)$ is a dcc-pair.]
\textbf{Case 1:} There is $g\in E(H')$ with $g\cap V(P)\neq \emptyset$ and $ |g\cap V(C)|\geq k-\ell+1$. By symmetry, we may assume
$u_1\in g$ and $|f_\ell\cap V(P)|\leq |f_1\cap V(P)|$. Since $\ell\leq r$, at most one $f_j$ is contained in $V(P)$.
So
\begin{equation}\label{l=r?}
 \mbox{\em $|f_\ell\cap V(C)|\geq r-\ell$ and if $r=\ell$, then  $|f_\ell\cap V(C)|\geq 1$.}
 \end{equation}
  By Claim~\ref{shortpath}, if $v_i \in g$ and $v_j \in f_\ell$, then $i=j$ or $|i-j| \geq \ell+1$. By Claim~\ref{ver-ver} with $I=g\cap V(C)$, $A=f_\ell\cap V(C)$ and $q=\ell+1$, we have $3$ possibilities. If $A=I$, then $c\geq (\ell+1)(k-\ell+1)\geq 2k$.
If $A\subsetneq I$ then $c\geq 2(k-\ell+1)+(\ell-1)(|f_\ell\cap V(C)|+1)$. By~\eqref{l=r?}, when $2\leq \ell\leq r-1$ this gives
$c\geq 2(k-\ell+1)+(\ell-1)(r-\ell+1)\geq 2k$, and when $\ell=r$ (and hence $k=r+1$) this gives
$c\geq 2(k-r+1)+(k-2)2= 2k$, again. If $A\setminus I\neq \emptyset$, then $c\geq 2(k-\ell+1) +|A|+2\ell-1>2k$. In all cases we get a contradiction.

\medskip
\textbf{Case 2:} There is $g\in E(H')$ such that $g\supseteq V(P)$. Then by Rule (S5), each $f_j\in E(P)$ contains $V(P)$. 
Since $|E(P)|\geq 2$, $\ell\leq r-1$. 
Let $E'=E(P)\cup \{g\}$. By Claim~\ref{cyc1}, every edge in $E'$ is a subset of $V(C) \cup V(P)$.  Set $I=( \bigcup_{f\in E'} f)\cap V(C)$. Since the sets $f_j\cap V(C)$ are distinct for distinct $j$, $|I|\geq r-\ell+1$. Moreover, if $r=\ell+1$ then $|I|\geq |E'|=\ell+1\geq 3>r-\ell+1$.
By the case, for any $u_i,u_{i'}\in V(P)$ and any $f,f'$, there is a $u_i,u_{i'}$-path of length $\ell-1$ whose set of edges is $E'-\{f,f'\}$. Hence if  $v_i\in f\in E'$ and $v_{i'}\in f'\in E'$ where $f'\neq f$, then $|i'-i|\geq \ell+1$. Since $|E'| \geq 3$, there are some $t \geq 3$ segments in $C$ of length at least $\ell+1$ that are internally disjoint from $I$. The vertices of $I$ are nonconsecutive on $C$ thus the edges $\{e_i, e_{i+1}: v_i \in I\}$ are distinct (but may intersect the segments of length at least $\ell+1$). We get
\[c\geq 2|I|+t(\ell-1)\geq 2(r-\ell+1)+t(\ell-1)=2r+\ell-1.\] This is at least $2k$, unless $t=3$, $k=r+1$, $\ell=2$ and $|I|=r-\ell+1$.
If $|I| = r-\ell+1$, then for each $f \in E', f \cap V(C)$ is an $r-\ell$-subset of $I$, and it follows that every distinct $v_i, v_j \in I$ satisfies $|i-j| \geq \ell+1$. In order to have $t = 3$, we must then have $|I|=3$. 
Thus, the unresolved situation is $\ell=2$, $|I|=3=r-2+1=r-1$ and $k=r+1=5$. In this case, $c=9$ and up to symmetry, $I=\{v_1,v_4,v_7\}$,
 so we may assume $f_1=\{u_1,u_2,v_1,v_4\},f_2=\{u_1,u_2,v_1,v_7\},g=\{u_1,u_2,v_4,v_7\}$. But $d_H(u_2)\geq 5$. If some $e_i$ contains $u_2$, then by symmetry we may  assume $i=1$, contradicting Claim~\ref{bigsmallcycle}.
Thus $d_C(u_2)=0$, so $d_{H'}(u_2)\geq 3$. One such edge is $g$, another maybe is $\{u_2,v_1,v_4,v_7\}$, but then the third must a vertex in $V(C)-\{v_1,v_4,v_7\}$, again contradicting Claim~\ref{bigsmallcycle} (maybe switching $g$ with $f_1$).

\medskip
\textbf{Case 3:} There are edges in $H'$ intersecting $V(P)$, but for each
 $g\in E(H')$,
   $|g\cap V(P)|\leq \ell-1$ and $ |g\cap V(C)|\leq k-\ell$. Then $k=r+1$ and for each
 $g\in E(H')$,
   $|g\cap V(P)|= \ell-1$. Let $A=\bigcup_{i=1}^\ell f_i\cap V(C)$. Since $|E(P)|\geq 2$, $|A|\geq r-\ell+1$.
 
 Suppose some distinct $i,j$, $v_i\in g$ and $v_{j}\in A$. Without loss of generality, $v_{j}\in f_\ell$ (because $P$ is a cycle). Since $|g\cap V(P)|= \ell-1$,
 we may also assume $u_1\in g$. By Claim~\ref{shortpath}, $|i-j| \geq \ell+1$.  
So, we can apply  Claim~\ref{ver-ver} with $I=g\cap V(C)$ and $q=\ell+1$. Since $|A|\geq |I|$, if $A\neq I$, then
  $$c\geq 2|I|+|A|+2q-3 \geq 2(r-\ell+1)+(r-\ell+1)+2\ell-1=3r+2-\ell\geq 2r+2, \mbox{a contradiction.}$$
Thus $A=I$ and the vertices of $A$ partition $V(C)$ into $r-\ell+1$ intervals of length at least $\ell+1$. 
If some 
$e_i$ contains a vertex of $P$, say $u_\ell\in e_i$, then $g\cap \{u_{\ell-1},u_1\}\neq \emptyset$, say $u_1\in g$.
So, in this case, if $v_{j}\in I$, then by Claim~\ref{shortpath}, $v_j$ is distance at least $\ell$ from $\{v_i, v_{i+1}\}$ in $C$. 
 An interval containing $m$ edges of $C$ that intersect $V(P)$ has at least $2\ell+m$ edges, and \[c\geq (\ell+1)(r-\ell+1)-(\ell+1)+2\ell +m \geq (\ell+1)(r-\ell+1) + \ell\] with equality only when exactly one edge of $C$ intersects $V(P)$. We get a contradiction, unless $\ell=r$ and exactly  one edge of $C$ intersects $V(P)$. In this case, $|I|=1$ and we can rename the vertices of $C$ so that $I=\{v_c\}$ and $e_r$ is the edge intersecting $V(P)$. Then each edge $f\in E(H')\cup E(P)$ intersecting $V(P)$ is contained in $V(P)\cup \{v_c\}$.
 So, the sum of degrees of vertices in $P$ is at most ${\ell \choose r-1}(r-1)\leq r(r-1)$ (for edges containing $v_c$) plus $r$ (for a possible edge $V(P)$) plus $r-2$ (for $e_r$). This totals $r^2+r-2$ which is less than $k\ell \leq \delta(H) |V(P)|$, a contradiction.

The remaining possibility is that no edges of $C$ intersect $V(P)$. Since $A=I$ for any choice of $g\in E(H')$ and $|I|=r-\ell+1$, 
$g\cap V(C)$ is the same for all  $g\in E(H')$ intersecting $V(P)$, and $f\cap V(C)\subset I$ for all $f\in E(P)$. So
each edge $f\in E(H')\cup E(P)$ intersecting $V(P)$ is contained in $V(P)\cup I$ and $|V(P)\cup I|=r+1$. Therefore,
$\sum_{u\in V(P)}d_H(u)$ is at most $\ell(r-\ell+1)$ (for the at most $r-\ell+1$ edges containing $V(P)$) plus
$\ell(\ell-1)$ (for the at most $\ell$ edges containing $I$), which sums to $\ell r<\ell k$, a contradiction.

\medskip
\textbf{Case 4:} No edges in $H'$ intersect $V(P)$ and $f_j\supseteq V(P)$ for each $j\in [\ell]$. Let $B(P)$ be the set of the edges in $C$ intersecting $V(P)$.
Since $|E(P)|\geq 2$, $\ell\leq r-1$. Let  $I= \bigcup_{i=1}^\ell f_i\cap V(C)$. By~\eqref{cyc1}, each $f_i$ is contained in $V(C) \cup V(P)$, therefore the sets $f_i\cap V(C)$ are distinct for distinct $i$. It follows that 
\begin{equation}\label{k-l}
|I|\geq r-\ell+1\geq k-\ell.
\end{equation}
By the case, for any $u_i,u_{i'}\in V(P)$ and any $f_j\in E(P)$, there is a $u_i,u_{i'}$-path of length $\ell-1$ whose set of edges is $E(P)-f_j$. Hence if  $v_j$ belongs to an edge of $E(P)$ and $e_{j'}\in B(P)$  where $j\neq j'$, then by Claim~\ref{shortpath} $v_j$ has distance at least $\ell$ from $\{v_{j'}, v_{j'+1}\}$. 
 Also the vertices of $I$ are nonconsecutive on $C$.  So, by Claim~\ref{ver-ed} with  $B=B(P)$ and $q=\ell$, we have 
$c\geq 2|I|+2(\ell-1)+|B(P)|.$ If $|I|\geq k-\ell+1$ or $|I|=k-\ell$ and $|B(P)|\geq 2$, then we get
 a contradiction. Otherwise, by~\eqref{k-l}, $k=r+1$, $|I|=k-\ell$ and $|B(P)|\leq 1$. Since each $u_j\in V(P)$ is in at least $k-\ell\geq 1$ edges of $C$, $|B(P)|= 1$, say $e_i \supseteq V(P)$, and $r-2=|e_i|-2 \geq |V(P)| = \ell$. So,
 $\sum_{u\in V(P)}d_H(u)  \leq \ell|E(P)|+\ell(1)=\ell(\ell+1)\leq \ell(r-1)<\ell k$, a contradiction.

\medskip
\textbf{Case 5:} No edges in $H'$ intersect $V(P)$ and there is $j\in [\ell]$ such that either $|f_j\cap V(C)|\geq k-\ell+1$ or 
 $|f_j\cap V(C)|= k-\ell$ and $|B_j\cup B_{j+1}|\geq 2$. By Claim~\ref{ver-ed} with $I=f_j\cap V(C)$
  $B=B_j\cup B_{j+1}$ and $q=\ell$, we have 
$c\geq 2|I|+2(\ell-1)+|B|.$ Since by the case $2|I|+|B|\geq 2k-2\ell+2$, we get a contradiction.

\medskip
\textbf{Case 6:} All other possibilities. This means  (a) no edges in $H'$ intersect $V(P)$, (b) 
there is $j_0\in [j]$ with $|f_{j_0}\cap V(P)|\leq \ell-1$,
 and (c) for each $j\in [\ell]$, $|f_j\cap V(C)|\leq k-\ell$, and   if
 $|f_j\cap V(C)|= k-\ell$ then  $|B_j\cup B_{j+1}|\leq 1$. Since $|f_{j_0}\cap V(C)|\geq r-(\ell-1)\geq k-\ell$, in order for (c) to hold, we need $k=r+1$ and  $|B_{j_0}\cup B_{j_0+1}|\leq 1$. For all $u_j \in V(P)$, $d_C(u_j) \geq k-\ell \geq 1$. In view of $u_{j_0}$, we need $k=\ell+1$.
 Thus $r=\ell=k-1$. Then as an edge in $C$ cannot contain all $\ell$ vertices in $P$, at least two edges of $C$ intersect $V(P)$, so there are two distinct $j_1,j_2$ such that 
 $|B_{j_i}\cup B_{j_i+1}|\geq 2$ for $i=1,2$. But since $\ell=r$, at most one of $f_{j_1}$ and $f_{j_2}$ contains $V(P)$, so the other satisfies Case 5, a contradiction.
\end{proof}

\section{Concluding remarks}

1. It would be interesting to understand whether for $4\leq k \leq r$, the bound on $\delta(H)$ in Theorem~\ref{mainthm3} may be lowered to $k-1$ to match  Construction~\ref{cons1} or there is a better construction.

2. Observe that there are sharpness examples for Theorems~\ref{dirac},~\ref{dirac2},~\ref{mainold2} and
\ref{mainthm} with connectivity $3$ and greater  (see, e.g. Construction~\ref{constbip} for  Theorem~\ref{mainthm}).
But we do not know $3$-connected  extremal examples  for  Theorem~\ref{mainthm}.
Moreover,  it well may be that hypergraphs with higher connectivity require a smaller minimum degree to force the existence of a long Berge cycle. A natural question to study would be the following.


{\bf Question.} Let $n \geq r+1 \geq k \geq 3$. What is the minimum number $f(k,r)$  such that every $n$-vertex, $r$-uniform, $3$-connected hypergraph with minimum degree at least $f(k,r)$ necessarily has $c(H) \geq \min\{n, |E(H)|, 2k\}$?

3. Another interesting question is: How the bound will change if instead of $r$-graphs we consider the hypergraphs in which the size of each edge is at least $r$?

\end{document}